\documentclass[letterpaper, 10 pt, conference]{ieeeconf}  

\IEEEoverridecommandlockouts                              
\usepackage{graphics} 
\usepackage{epsfig} 
\usepackage{amsmath} 
\usepackage{amssymb}  

\makeatletter
\let\proof\@undefined                        
\let\endproof\@undefined                  
\makeatother
\usepackage{amsthm}
\usepackage{multirow}
\usepackage{bbm, dsfont}
\usepackage{bm} 
\usepackage{url}
\usepackage[]{algorithm2e}
\usepackage{balance}
\usepackage{color}
\usepackage{subfigure}
\graphicspath{{./Figures/}}
\usepackage{array}
\newcolumntype{L}[1]{>{\raggedright\let\newline\\\arraybackslash\hspace{0pt}}m{#1}}
\newcolumntype{C}[1]{>{\centering\let\newline\\\arraybackslash\hspace{0pt}}m{#1}}
\newcolumntype{R}[1]{>{\raggedleft\let\newline\\\arraybackslash\hspace{0pt}}m{#1}}

\newcommand{\R}{\mathbb{R}}

\newtheorem{theorem}{Theorem}
\newtheorem{definition}{Definition}
\newtheorem{assumption}{Assumption}
\newtheorem{lemma}{Lemma}

\newtheorem{remark}{Remark}

\newtheorem{problem}{Problem}

\newcommand{\barr}[1]{\overline{#1}}
\newenvironment{rcases}
{\left.\begin{aligned}}
	{\end{aligned}\right\rbrace}

\DeclareMathOperator*{\diag}{diag}
\DeclareMathOperator*{\vect}{vec}

\title{\LARGE \bf
Optimal Input Design for Affine Model Discrimination with Applications in Intention-Aware Vehicles}

\author{Yuhao Ding$^{\ast}$\;\;\; Farshad Harirchi$^{\ast}$\;\;\; Sze Zheng Yong$^{\ast}$\;\;\; Emil Jacobsen\;\;\; Necmiye Ozay
	\thanks{$^{\ast}$ Equal contribution from these authors.}
	\thanks{Y. Ding, F. Harirchi and N. Ozay are with the Electrical Engineering and Computer Science Department, University of Michigan, Ann Arbor, MI, 48109 (emails: {\tt {\{yuhding,harirchi,necmiye\}@umich.edu}}).  S.Z. Yong is with the School for Engineering of Matter, Transport and Energy, Arizona State University, Tempe, AZ, 85287 (email: {\tt {szyong@asu.edu}}). E. Jacobsen is with the Department of Mathematics, KTH Royal Institute of Technology, Stockholm, Sweden (email: {\tt emiljaco@kth.se}).}
	\thanks{This work was supported by an Early Career Faculty grant from
NASA's Space Technology Research Grants Program, DARPA grant N66001-14-1-4045 and by Toyota Research Institute (``TRI"). TRI provided funds to assist the authors with their research but this article solely reflects the opinions and conclusions of its authors and not TRI or any other Toyota entity.}
}

\begin{document}

\maketitle
\thispagestyle{empty}
\pagestyle{empty}
\begin{abstract}
	This paper considers the optimal design of input signals for the purpose of discriminating among a finite number of affine models with uncontrolled inputs and noise. Each affine model represents a different system operating mode, corresponding to unobserved intents of other drivers or robots, or to fault types or attack strategies, etc. The input design problem aims to find optimal separating/discriminating (controlled) inputs such that the output trajectories of all the affine models are guaranteed to be distinguishable from each other, despite uncertainty in the initial condition and uncontrolled inputs as well as the presence of process and measurement noise. We propose a novel formulation to solve this problem, with an emphasis on guarantees for model discrimination and optimality, in contrast to a previously proposed conservative formulation using robust optimization. This new formulation can be recast as a bilevel optimization problem and further reformulated as a mixed-integer linear program (MILP). Moreover, our fairly general problem setting allows the incorporation of  objectives and/or responsibilities among rational agents. For instance, each driver has to  obey traffic rules, while simultaneously optimizing for safety, comfort and energy efficiency. Finally, we demonstrate the effectiveness of our approach for identifying the intention of other vehicles in several driving scenarios. 
	
\end{abstract}

\section{Introduction} \label{sec:intro}
As cyber-physical systems become increasingly complex, integrated and interconnected, they inevitably have to interact with other systems under partial knowledge of each other's internal state such as intention and mode of operation. For instance, autonomous vehicles and robots must operate without access to the intentions or decisions of nearby vehicles or humans \cite{Sadigh2016,Yong2014}.  Similarly, system behaviors change in the presence of different fault types \cite{Harirchi2015Model,Cheong2015} or attack modes \cite{Pasqualetti2013,Yong2015}, and the true system model is often not directly observed. In both examples, there is a number of possible system behaviors and the objective is to develop methods for discriminating among these models (of system behaviors) based on noisy observed measurements. This is an important problem in statistics, machine learning and systems theory; thus, general techniques for model discrimination can have a significant impact on a broad range of applications. 


\noindent\emph{1) Literature Review:}
The problem of discriminating among a set of linear models appears in a wide variety of research areas such as fault detection, input-distinguishability and mode discernibility of hybrid systems. Approaches in the literature can be categorized into passive and active methods. 
While passive techniques seek the separation of the models regardless of the input \cite{Lou2009Distinguishability,Rosa2011Distinguishability,Harirchi2016Guaranteed,Harirchi2017}, active methods search for an input such that the behaviors of different models are distinct. The focus of this paper is on active model discrimination techniques. The concept of distinguishability of two linear time invariant models was defined in \cite{Grewal1976Identifiability}, while similar concepts were considered in the hybrid systems community for mode discernibility/distinguishability \cite{Babaali2004Observability,Desantis2016Observability}. The problem of model-based active fault detection was also extensively studied, where the goal is to find a small excitation that has a minimal effect on the desired behavior of the system, while guaranteeing the isolation of different fault models \cite{Cheong2015,Simandl2009Active,Nikoukhah2006Auxiliary,Scott2014Input}. 

Specifically in the area of intention identification, passive methods have been investigated in \cite{Yong2014,Ziebart2009Planning,Nguyen2012Capir,Kumar2013Learning} to estimate human behavior or intent,  and the obtained intention estimates are then used for control. The problem of intention identification was also considered for inter-vehicle applications, where a partially observable
Markov decision process (POMDP) framework was proposed to estimate the driver's intention \cite{Lam2014Pomdp}. On the other hand, an active approach for identifying human intention in human-autonomous vehicle interactions 
 was studied in \cite{Sadigh2016}, where the intentions of human drivers were estimated by applying a wide variety of inputs to the autonomous car and by observing the reactions of the human drivers to those excitations.
\\
\noindent\emph{2) Main Contributions:}
{We propose a novel optimization-based method for non-conservative \emph{active} model discrimination through the design of optimal separating/discriminating inputs, which improves on a previous robust optimization formulation \cite{harirchi2017active}, in which the solution is conservative (i.e., the obtained input sequence is still separating but incurs a higher cost), if it exists. 
As in \cite{harirchi2017active}, our formulation applies to a general class of affine models with uncontrolled inputs that extends the class of models considered in \cite{Nikoukhah2006Auxiliary,Scott2014Input}, 
and considers a principled characterization of input and state constraints based on 
responsibilities 
of rational agents. 

We show that the active model discrimination problem can be recast as a bilevel optimization problem, which can then be reformulated as a tractable MILP problem (MIQP, in the case of a quadratic objective function) with the help of Special Ordered Set of degree 1 (SOS-1) constraints \cite{Beale1976Global}. Our approach guarantees the separation among affine models, while being optimal as opposed to the previously proposed conservative approach, but at the cost of higher computational complexity. 
A further feature of our approach is that we do not require or compute an explicit set representation of the states in contrast to existing approaches that consider polyhedral projections \cite{Nikoukhah2006Auxiliary} and zonotopes  \cite{Scott2014Input}, which are known to be rather limiting. Finally, we compare our proposed approach with the conservative approach in \cite{harirchi2017active} in terms of computational complexity and optimality using examples of intention identification of other human-driven or autonomous vehicles approaching at an intersection and during a lane change on a highway. 

\section{Preliminaries} \label{sec:prelim2}
In this section, we introduce some notations and definitions, and describe the modeling framework we consider.
\subsection{ Notation and Definitions}
Let $ x \in \R^n$ denote a vector and $ M \in \R^{n\times m}$ a matrix, with transpose $M^\intercal$ and $M \geq 0$ denotes element-wise non-negativity. The vector norm of  $x$ is denoted by $\| {x} \|_i$ with $i \in \{1,2,\infty\}$, while $\mathbf{0}$, $\mathbbm{1}$ and $\mathbbm{I}$ represent the vector of zeros, the vector of ones and the identity matrix of appropriate dimensions. 
The $\diag$ and $\vect$ operators are defined for a collection of matrices $M_i,i=1,\dotsc,n$ and matrix $M$ as: 
\begin{equation*}
\begin{aligned} &\hspace{-0.3cm} \textstyle\diag_{i=1}^n \{ M_i\} = \begin{bmatrix} M_1 & &\\ & \ddots & \\ & & M_n \end{bmatrix},  & \hspace{-0.2cm} \textstyle\vect_{i=1}^n \{ M_i\} = \begin{bmatrix} M_1 \\ \vdots \\ M_n \end{bmatrix},  \\
&\hspace{-0.3cm} \textstyle\diag_{i,j} \{ M_k\} = \begin{bmatrix} M_i & \mathbf{0} \\  \mathbf{0} & M_j \end{bmatrix},  & \hspace{-0.2cm} \textstyle\vect_{i,j} \{ M_k\} = \begin{bmatrix} M_i  \\ M_j \end{bmatrix},  \\
	& \hspace{-0.3cm}\textstyle\diag_N \{ M\} = \mathbbm{I}_N \otimes M,  & \hspace{-0.2cm}  \textstyle\vect_N \{M\} = \mathbbm{1}_N \otimes M,
	\end{aligned}
\end{equation*}

\noindent where $\otimes$ is the Kronecker product. 

The set of positive integers up to $n$ is denoted by $\mathbb{Z}_n^{+} $, and the set of non-negative integers up to $n$ is denoted by $\mathbb{Z}_n^{0}$.   We will also make use of  Special Ordered Set of degree 1 (SOS-1) constraints\footnote{Off-the-shelf solvers such as Gurobi and CPLEX \cite{gurobi,cplex} can readily handle these constraints, which can significantly reduce the search space for integer variables in branch and bound algorithms.} in our optimization formulations, \hspace{-0.05cm}defined as: 
\begin{definition}[SOS-1 Constraints \cite{gurobi}] \label{def:SOS1}
	A special ordered set of degree 1 (SOS-1) constraint is a set of integer, continuous or mixed-integer scalar variables for which at most one variable in the set may take a value other than zero, denoted as SOS-1: $\{v_1,\hdots,v_N\}$. For instance, if $v_i \neq 0$, then this constraint imposes that $v_j=0$ for all $j \neq i$.
\end{definition}
\subsection{Modeling Framework} \label{sec:model}
Consider $N$ discrete-time affine time-invariant models $\mathcal{G}_i = (A_i,B_i,B_{w,i},C_i,D_i,D_{v,i},f_i,g_i)$, each with states $\bm{\vec{x}}_i \in \mathbb{R}^n$, outputs $z_i \in \mathbb{R}^p$, inputs $\bm{\vec{u}}_i \in \mathbb{R}^m$, process noise $w_i \in \mathbb{R}^{m_w}$ and measurement noise $v_i \in \mathbb{R}^{m_v}$. The models evolve according to the state and output equations:
\begin{align}
	\label{eq:state_eq}
	\bm{\vec{x}}_i(k+1) &= A_i \bm{\vec{x}}_i(k) + B_i\bm{\vec{u}}_i(k) + B_{w,i} {w}_i(k)+f_i,\\
	z_i(k) &= C_i\bm{\vec{x}}_i(k) + D_i\bm{\vec{u}}_i(k) +D_{v,i} {v}_i(k)+g_i. \label{eq:output_eq}
	\end{align}

The initial condition for model $i$, denoted by $\bm{\vec{x}}^0_i=\bm{\vec{x}}_i(0)$, is constrained to a polyhedral set with $c_0$ inequalities:
\begin{align}
\label{eq:bmx0_polytope}
\bm{\vec{x}}^0_i \in \mathcal{X}_0 &= \{\bm{\vec{x}}\in\R^n: P_0\bm{\vec{x}} \leq p_0\}, \; \forall i \in \mathbb{Z}_N^+.
\end{align}
%
The first $m_u$ components of $\bm{\vec{u}}_i$ are controlled inputs, denoted as $u \in \mathbb{R}^{m_u}$, which are equal for all $\bm{\vec{u}}_i$, while the other $m_d=m-m_u$ components of $\bm{\vec{u}}_i$, denoted as $d_i \in \mathbb{R}^{m_d}$, are uncontrolled inputs that are model-dependent. Further, the states $\bm{\vec{x}}_i$ are partitioned into $x_i \in \mathbb{R}^{n_x}$ and $y_i \in \mathbb{R}^{n_y}$, where $n_y=n-n_x$, as follows:  
\begin{equation}
	\label{eq:bold}
	\bm{\vec{u}}_i(k)=\begin{bmatrix} u(k) \\ d_i(k) \end{bmatrix}, \bm{\vec{x}}_i(k)=\begin{bmatrix} x_i(k) \\ y_i(k) \end{bmatrix}.
	\end{equation}

The states $x_i $ and $y_i$ represent the subset of the states $\bm{\vec{x}}_i$ that are the `responsibilities' of the controlled and uncontrolled inputs, $u$ and $d_i$, respectively. The term `responsibility' in this paper is to be interpreted as $u$ and $d_i$, respectively, having to independently satisfy the following polyhedral state constraints (for $k \in \mathbb{Z}_T^{+}$) with $c_x$ and $c_y$ inequalities:
\begin{align}
\label{eq:x_polytope}
x_i(k) \in \mathcal{X}_{x,i} &= \{ x\in\R^{n_x}: P_{x,i}x \leq p_{x,i}\}, \\
\label{eq:y_polytope}
y_i(k) \in \mathcal{X}_{y,i} &= \{ y\in\R^{n_y}: P_{y,i}y \leq p_{y,i}\},
\end{align}
subject to constrained inputs described by polyhedral sets (for $k \in \mathbb{Z}_{T-1}^{0}$) with $c_u$ and $c_d$ inequalities, respectively:
\begin{align}
\label{eq:u_polytope}
u(k) \in \mathcal{U} &= \{ u\in\R^{m_u}: Q_uu \leq q_u\}, \\
\label{eq:d_polytope}
{d}_i(k) \in \mathcal{D}_i &= \{ d\in\R^{m_{d_i}}: Q_{d,i}d \leq q_{d,i}\}.
\end{align}

On the other hand, the process noise  $w_i$  and measurement noise $v_i$ 
are also polyhedrally constrained with $c_w$ and $c_y$ inequalities, respectively:
\begin{align}  
\label{eq:w_polytope}
w_i(k) \in \mathcal{W}{_i} &= \{ w\in\R^{m_w}: Q_{w,{i} } w\leq q_{w,i}\},\\
\label{eq:v_polytope}
v_i(k) \in \mathcal{V}{_i}  &= \{ v\in\R^{m_v}: Q_{v,{i} } v \leq q_{v,i}\},
\end{align}
and have no responsibility to satisfy any state constraints.

Using the above partitions of states and inputs, the corresponding partitioning of the state and output equations in \eqref{eq:state_eq} and \eqref{eq:output_eq} are:

{\small \vspace*{-0.4cm}\begin{align}
\label{eq:state_eq_2}
&\begin{array}{rl}
\bm{\vec{x}}_i(k+1) &=  \begin{bmatrix}
  A_{xx,i} \quad A_{xy,i} \\
  A_{yx,i} \quad A_{yy,i}
\end{bmatrix}\bm{\vec{x}}_i(k) + \begin{bmatrix}
  B_{xu,i} \quad B_{xd,i} \\
  B_{yu,i} \quad B_{yd,i}
\end{bmatrix}\bm{\vec{u}}_i(k) \\
&\quad +\begin{bmatrix}
B_{xw,i} \\
B_{yw,i}
\end{bmatrix}w_i(k) + \begin{bmatrix} f_{x,i} \\ f_{y,i} \end{bmatrix}.
\vspace*{-0.4cm}\end{array}\\
&z_i(k)= C_i \bm{\vec{x}}_i(k) + \begin{bmatrix} D_{u,i} \quad D_{d,i} \end{bmatrix} \bm{\vec{u}}_i(k) +D_{v,i}v_{i}(k)+ g_i.
\end{align}}\normalsize
\noindent Further, we will consider a time horizon of length $T$ and introduce some time-concatenated notations. The time-concatenated states and outputs are defined as 
\begin{align*}
\bm{\vec{x}}_{i,T} &= \textstyle\vect_{k=0}^T \{ \bm{\vec{x}}_{i}(k)\}, \quad x_{i,T}= \textstyle\vect_{k=0}^T \{ {x}_{i}(k)\},\\ y_{i,T}&= \textstyle\vect_{k=0}^T \{ {y}_{i}(k)\}, \quad z_{i,T}= \textstyle\vect_{k=0}^T \{ {z}_{i}(k)\},
\end{align*}
while the time-concatenated inputs and noises are defined as

{\small \vspace*{-0.4cm}\begin{align*}
 \bm{\vec{u}}_{i,T} \hspace{-0.09cm}&=\hspace{-0.09cm} \textstyle\vect_{k=0}^{T-1} \hspace{-0.05cm} \{ \bm{\vec{u}}_{i}(k)\}, u_{T}\hspace{-0.09cm}=\hspace{-0.09cm} \textstyle\vect_{k=0}^{T-1} \hspace{-0.05cm}\{ {u}(k)\}, d_{i,T} \hspace{-0.09cm}=\hspace{-0.09cm} \textstyle\vect_{k=0}^{T-1} \hspace{-0.05cm}\{ {d}_{i}(k)\},\\
	{w}_{i,T} \hspace{-0.09cm}&=\hspace{-0.09cm} \textstyle\vect_{k=0}^{T-1} \hspace{-0.05cm} \{ {w}_{i}(k)\}, {v}_{i,T} \hspace{-0.09cm}=\hspace{-0.09cm} \textstyle\vect_{k=0}^{T} \hspace{-0.05cm} \{ {v}_{i}(k)\}.
	\end{align*}}\normalsize

Given $N$ discrete-time affine models, there are $I={N \choose 2}$ model pairs  and 
let the mode $\iota\in \{1,\cdots, I \}$ denote the pair of models $(i,j)$. Then, concatenating $ \bm{\vec{x}}_{i}^0$ ,$\bm{\vec{x}}_{i,T}$, ${x}_{i,T}$, $y_{i,T}$, $d_{i,T}$ , $z_{i,T}$, ${w}_{i,T}$ and $ {v}_{i,T}$ for each model pair, we define
\begin{align*}
\nonumber& \bm{\vec{x}}_0^\iota=\textstyle\vect_{i,j}\{ \bm{\vec{x}}_{i}^0\},\; \; \bm{\vec{x}}_T^\iota=\textstyle\vect_{i,j}\{ \bm{\vec{x}}_{i,T}\}, \;\;\bm{\vec{u}}_T^\iota=[{u}_T^{\mathsf{T}},{d}_T^{\iota \mathsf{T}}]^\mathsf{T},\\
\nonumber& {x}_T^\iota=\textstyle\vect_{i,j} \{{x}_{i,T}\},\; 
{y}_T^\iota=\textstyle\vect_{i,j} \{ {y}_{i,T}\}, \; \; {z}_T^\iota=\textstyle\vect_{i,j} \{ {z}_{i,T}\},\\
\nonumber & {d}_T^\iota=\textstyle\vect_{i,j} \{ {d}_{i,T}\}, \; \;  {w}_T^\iota=\textstyle\vect_{i,j} \{ {w}_{i,T}\}, \; {v}_T^\iota=\textstyle\vect_{i,j} \{ {v}_{i,T}\}.
\end{align*}

The states and outputs over the entire time horizon for each mode $\iota$ can be written as simple functions of the initial state $\bm{\vec{x}}_0^\iota$, input vectors ${u}_T$, ${d}_T^\iota$, and noise $ {w}_T^\iota$, $ {v}_T^\iota$:
\begin{align}
\label{eq:x}
x_T^\iota&= M_x^\iota \bm{\vec{x}}_0^\iota + \Gamma_{xu}^\iota u_T + \Gamma_{xd}^\iota d_T^\iota +\Gamma_{xw}^\iota w_T^\iota+ \tilde{f}_x^\iota, \\
\label{eq:y}
y_T^\iota &= M_y^\iota \bm{\vec{x}}_0^\iota + \Gamma_{yu}^\iota u_T + \Gamma_{yd}^\iota d_T^\iota +\Gamma_{yw}^\iota w_T^\iota+ \tilde{f}_y^\iota, \\
\label{eq:bmx}
\bm{\vec{x}}_T^\iota &= \bar{A}^\iota \bm{\vec{x}}_0^\iota + \Gamma_u^\iota u_T + \Gamma_d^\iota d_T^\iota +\Gamma_w^{\iota}w_T^\iota+ \tilde{f}^\iota, \\
\label{eq:z}
z_T^\iota & = \bar{C}^\iota \bm{\vec{x}}_T^\iota+\bar{D}_u^\iota {u}_T+\bar{D}_d^\iota {d}_T+\bar{D}_v^\iota {v}_T^\iota+ \tilde{g}^\iota.
\end{align}

The matrices and vectors $M_\star^\iota$, $\Gamma_{\star u}^\iota$, $\Gamma_{\star w}^\iota$, $\Omega_\star^\iota$ and $\tilde{f}_\star^\iota$ for $\star \in \{x,y\}$, and $\bar{A}^\iota$, $\Gamma_u^\iota$, $\Gamma_d^{q}$, $\Omega^\iota$, $\bar{C}^\iota$, $\bar{D}^\iota$, $\bar{D}_v^\iota$, $\tilde{f}^\iota$,$\tilde{g}^\iota$ are defined in the appendix. Moreover, the uncertain variables for each mode $\iota$ are concatenated as 
$\overline{x}^\iota=[\bm{\vec{x}}_0^{\iota\mathsf{T}} \ d_T^{^\iota\mathsf{T}} \ w_T^{^\iota\mathsf{T}} \ v_T^{^\iota\mathsf{T}}]^\mathsf{T}$.

We then concatenate the polyhedral state constraints in \eqref{eq:x_polytope} and \eqref{eq:y_polytope}, eliminating $x_T$ and $y_T$ in them and expressing them in terms of  $\bar{x}^\iota$ and ${u}_T$. First, let

{\vspace*{-0.4cm}\begin{align*}
	\bar{P}_x^\iota &= \textstyle\diag_{i,j}\diag_T\{ P_{x,i}\}, \quad \bar{P}_y^\iota = \diag_{i,j}\diag_T\{ P_{y,i}\},\\
	\bar{p}_x^\iota &= \textstyle\vect_{i,j}\vect_T\{ p_{x,i}\}, \ \quad \bar{p}_y^\iota = \vect_{i,j}\vect_T\{ p_{y,i}\}.
	\end{align*}}\normalsize

Then, 
we can rewrite the polyhedral constraints on $\star$ as:

{ \vspace*{-0.4cm}\begin{align*}
	\bar{P}_{\star}^\iota x_T^\iota \leq \bar{p}^\iota_{\star} &\Leftrightarrow
	H_{\star}^\iota \bar{x}^\iota \leq h^\iota_{\star}(u_T), \; \star \in \{x,y \}
	\end{align*}}\normalsize
where { $
	H_{\star} ^\iota= \bar{P}^\iota_{\star}\begin{bmatrix} M^\iota_{\star} & \Gamma^\iota_{\star d}& \Gamma^\iota_{\star w}&\bm{0} \end{bmatrix}$ and $ 
	h_{\star}^\iota(u_T) = \bar{p}^\iota_{\star}-\bar{P}^\iota_{\star}\Gamma^\iota_{\star u}u_T^\iota-\bar{P}^\iota_{\star}\tilde{f}^\iota_{\star}.
	$}\normalsize \ Similarly, let

{ \vspace*{-0.4cm}\begin{align*} 
	\bar{Q}_{u} &=\textstyle\diag_T\{ Q_{u}\}, \; \;
	\bar{Q}^\iota_{\dagger} = \textstyle\diag_{i,j}\textstyle \diag_T\{ Q_{{\dagger},i}\},\\
	\bar{q}_u &= \textstyle\vect_T\{ q_u\}, \;\;
	\bar{q}^\iota_{\dagger} = \textstyle\vect_{i,j}\textstyle\vect_T\{ q_{{\dagger},i}\}, \ \dagger \in \{d,w,v \}.
	\end{align*}}\normalsize

Then, the polyhedral input constraints in \eqref{eq:u_polytope} and \eqref{eq:w_polytope} 
for all $k$ are equivalent to $\bar{Q}_u u_T \leq \bar{q}_u$ and $\bar{Q}^{\iota}_{\dagger} {\dagger}^{\iota}_T \leq \bar{q}^{\iota}_{\dagger}$.

Moreover, we concatenate the initial state constraint in \eqref{eq:bmx0_polytope}:
{ \vspace*{-0.4cm}\begin{align*} 
\bar{P}_{0} ^\iota=\textstyle\diag_2\{ P_{0}\}, \quad \bar{p}_{0}^\iota=\textstyle\vect_2\{ p_{0}\}, \quad
\end{align*}}\normalsize
 Hence, in terms of $\bar{x}^{\iota}$, we have a polyhedral constraint of the form $H_{\bar{x}}^{\iota}\bar{x}^\iota \leq h_{\bar{x}}^{\iota}$, with\\
{ \vspace*{-0.4cm}\begin{align*}
 H_{\bar{x}}^{\iota} = \begin{bmatrix} \bar{P}_{0} ^\iota& 0 &0&0 \\ 0 & \bar{Q}_d^{\iota}&0&0\\ 0&0& \bar{Q}_w^{\iota}&0\\ 0 &0&0&\bar{Q}_v^{\iota}\end{bmatrix},
h_{\bar{x}}^{\iota} = \begin{bmatrix} \bar{p}_{0} ^\iota\\ \bar{q}_d^{\iota}\\ \bar{q}_w^{\iota}\\ \bar{q}_v^{\iota} \end{bmatrix}.
	\end{align*}}\normalsize

We also use the notations without superscript $\iota$, e.g., $H_{\bar{x}}$ and $h_{\bar{x}}$, to denote the variables, matrices and vectors that are concatenated across all models, with $\vect_{i,j}$ and $\diag_{i,j}$ being replaced by $\vect_{i=1}^N$ and $\diag_{i=1}^N$.

\begin{remark}
Since it is the responsibility of $d_i$ to satisfy the constraint in \eqref{eq:y_polytope}, it is important to make sure that the models are meaningful in the sense that for the range of time horizons of interest, $T$,  and for each $i \in \mathcal{Z}_N^+$,
\begin{align}
	\exists {d}_{i}(k) \in \mathcal{D}_i, \forall k \in \mathbb{Z}_{T-1}^0 : \eqref{eq:y_polytope} \ \text{\normalsize is satisfied}
\end{align}
\noindent for any given $\bm{\vec{x}}_0 \in \mathcal{X}_0$ that satisfies $y_i(0) \in \mathcal{X}_{y,i}$ (cf.  \eqref{eq:y_polytope}) for all $i \in \mathbb{Z}_N^+$, and for any given $u(k) \in \mathcal{U}$ for all $k \in \mathbb{Z}_{T-1}^0$. Similarly, the constraint in \eqref{eq:x_polytope} must be satisfiable by $u$ for any uncertainties.
If the considered affine model satisfies these assumptions, we refer to it as  \emph{well-posed}. Note that models that do not satisfy this assumption are impractical, since the responsibilities of the  inputs will be impossible to be satisfied; thus, we shall assume throughout the paper that the given affine models are always well-posed. 
\end{remark}

\begin{remark}
The case without `responsibilities' is a special case of the above modeling framework with 
$n_y=n_{d}=0$. 
\end{remark}

\section{Problem Formulation} \label{sec:form2}
\subsection{Active Model Discrimination Problem}
Designing a separating input for model discrimination is equivalent to finding an admissible input for the system, such that if the system is excited with this input, any observed trajectory is consistent with only one model, regardless of any realization of uncertain parameters. In addition, the designed separating input must be optimal for a given cost function $J(u_T)$. The problem of input design for model discrimination can be defined formally as follows:
\begin{problem}[Exact Active Model Discrimination] \label{prob:1}
Given $N$ well-posed affine models $\mathcal{G}_i$, and state, input and noise constraints, 
\eqref{eq:bmx0_polytope}, \eqref{eq:y_polytope}, \eqref{eq:d_polytope}-\eqref{eq:v_polytope},  
 find an optimal input sequence $u^*_T$  to minimize a given cost function $J(u_T)$    such that  for all possible initial states $\bm{x}_0$, uncontrolled inputs ${d}_T$, process noise  $w_T$ and measurement noise $v_T$, only one model is valid, i.e., the output trajectories of any pair of models have to differ {by a threshold $\epsilon$} in at least one time instance. The optimization problem can be formally stated as follows:
	\begin{subequations} \label{eq:quant}
		\begin{align}
		\nonumber \min_{u_T,x_T,{z_T}}  J(u_T) &  \\
		\text{s.t. }\hspace{0.2cm}
		\forall k \in \mathbb{Z}_{T-1}^0:&\ \text{\eqref{eq:u_polytope} holds},\\
		\hspace{-0.6cm} \begin{array}{l} \begin{rcases}
		\forall i,j \in \mathbb{Z}_\mathrm{N}^{+}, i<j, \forall k \in \mathbb{Z}_T^{0},\\
		\forall   \bm{x}_0,y_T ,d_T,w_T, v_T:\\ \text{\eqref{eq:state_eq}-\eqref{eq:bmx0_polytope},\eqref{eq:y_polytope},\eqref{eq:d_polytope}-\eqref{eq:v_polytope} hold} \end{rcases} \end{array} \hspace{-0.2cm}
		:& \begin{array}{l} \{ \forall k{^{\prime}} \in \mathbb{Z}_T^{+}: \text{\eqref{eq:x_polytope} holds}\}
		\wedge  \\
		\{\exists k{^{\prime}}\in \mathbb{Z}_T^{0}, \\
		{|z_{i}(k{^{\prime}}) - z_{j}(k{^{\prime}})| \geq \epsilon} \}. \end{array} \label{eq:separability_logic1}\hspace{-0.55cm}
		\end{align}
	\end{subequations}
\end{problem}


The first predicate in \eqref{eq:separability_logic1} guarantees that the `responsibility' of the controlled input is satisfied, while the second predicate is the separation condition, which ensures that for each pair of models and for all possible values of the uncertain variables, there must exist at least one time instance such that the output values of two models are different. The latter means that we are first dealing with the uncertainties, then considering the quantifier on $k$ (time instance) for each uncertainty. If we change the order of quantification as was done in \cite{harirchi2017active}, by first considering the existence quantifier and then dealing with all uncertainties, a conservative active model discrimination approach will be obtained.

\begin{problem}[Conservative Active Model Discrimination \cite{harirchi2017active}] \label{prob:2}
	Given $N$ well-posed affine models $\mathcal{G}_i$, and state, input and noise constraints 
	\eqref{eq:bmx0_polytope},\eqref{eq:y_polytope},\eqref{eq:d_polytope}-\eqref{eq:v_polytope},  find an optimal input sequence $u^*_T$  to minimize a given cost function $J(u_T)$    such that there exists at least one time instance at which the output trajectories of each pair of models are different {by a threshold $\epsilon$} for all possible initial states $\bm{x}_0$, uncontrolled inputs ${d}_T$, process noise  $w_T$ and measurement noise $v_T$. The optimization problem can be formally stated as follows:
		\begin{subequations} \label{eq:quant2}
		\begin{align}
		\nonumber \min_{u_T,x_T,{z_T}}  J(u_T) &  \\
		\text{s.t. }\hspace{0.2cm}
		\forall k \in \mathbb{Z}_{T-1}^0:&\ \text{\eqref{eq:u_polytope} holds},\\
		\hspace{-0.5cm}\ \begin{array}{l}{\forall i,j \in \mathbb{Z}_\mathrm{N}^{+}, i<j,}\\
		{\exists k^{\prime}}\in \mathbb{Z}_T^{0}, \\
			{|z_{i}(k{^{\prime}}) - z_{j}(k{^{\prime}})| \geq \epsilon}\\
		\end{array} : 
		&
		\begin{array}{l} \begin{cases}
	 \forall k \in \mathbb{Z}_T^{0},\\
		\forall   \bm{x}_0,y_T ,d_T,w_T, v_T:\\ \text{\eqref{eq:state_eq}-\eqref{eq:bmx0_polytope},\eqref{eq:x_polytope},\eqref{eq:y_polytope},\eqref{eq:d_polytope}-\eqref{eq:v_polytope} hold.} \end{cases} \end{array} \label{eq:separability_logic2} \hspace{-0.6cm}
		\end{align}
	\end{subequations}
\end{problem} \vspace*{-0.15cm}

Note that the quantifier order matters. In the first optimal formulation in Problem \ref{prob:1}, the `for all' quantifier precedes the `there exists' quantifier, implying that the time instance at which separation is enforced can be dependent on the realization of the uncertain variables (similar to adjustable robust optimization \cite{ben2004adjustable}). In the latter formulation in Problem \ref{prob:2} \cite{harirchi2017active}, the `there exists' quantifier precedes the `for all' quantifier, thus separation is enforced at the same time instance for all realizations of the uncertain variables. 

\section{Active Model Discrimination Approach} \label{sec:method}
In this section, we propose an optimization-based approach to solve Problem \ref{prob:1} and briefly review the solution that is proposed in \cite{harirchi2017active} to solve Problem \ref{prob:2}, against which we will compare our novel exact approach in the next section. Specifically, we show that Problem \ref{prob:1} can be posed as a bi-level optimization problem that can be further converted to a single level optimization problem using KKT conditions. As we illustrate in a comparison study, this problem can be computationally demanding compared to the conservative approach proposed in \cite{harirchi2017active}. However, this approach delivers the optimal separating input that can potentially have significantly better performance when compared to \cite{harirchi2017active}. 

\subsection{Exact Active Model Discrimination Approach} \label{sec:exact}
		
First, we show in the following lemma that 
Problem \ref{prob:1} can be reformulated as a bilevel optimization problem. 
\begin{lemma}[Bilevel Optimization Formulation] \label{prop:bilevel1}
	Given a separability index $\epsilon$, the active model discrimination problem in Problem \ref{prob:1} is equivalent 
	to a bilevel optimization problem with the following outer problem:	
	\begin{subequations}
			\begin{align}
			\tag{$P_{Outer}$}&\min_{u_T} J(u_T) \hspace*{1cm}\\
			\text{s.t. \; }  
			\forall i\in \mathbb{Z}_N^{+},\forall k \in \mathbb{Z}_{T-1}^0:&\ \text{\eqref{eq:u_polytope} holds},\\
			\hspace{-0.6cm} \begin{array}{l} \begin{rcases}
		\forall i,j \in \mathbb{Z}_\mathrm{N}^{+}, i<j, \forall k \in \mathbb{Z}_T^{0},\\
		\forall   \bm{x}_0,y_T ,d_T,w_T, v_T:\\ \text{\eqref{eq:state_eq}-\eqref{eq:bmx0_polytope},\eqref{eq:y_polytope},\eqref{eq:d_polytope}-\eqref{eq:v_polytope} hold} \end{rcases} \end{array} \hspace{-0.2cm}
		:& \begin{array}{l} \forall k \in \mathbb{Z}_T^{+}: 
		\text{\eqref{eq:x_polytope} holds}, \end{array} \label{eq:x_resp} \\ 
			\forall \iota \in \mathbb{Z}_\mathrm{I}^{+} :& \ \delta^{\iota*}(u_T) \geq \epsilon, \label{eq:eps}
			\end{align} \label{eq:20}
		\end{subequations}
		where $\delta^{\iota*}(u_T)$ is the solution to the inner problem: 
\begin{subequations}
				\begin{align}
				\tag{$P_{Inner}$}\delta^{{\iota}*}(u_T)=&\min_{\delta^{\iota},\bm{x}_0^\iota,d_T^{\iota},w_T^{\iota}, v_T^{\iota}}  \delta^{\iota} \\
				\text{s.t. \; }
				\forall i\in \mathbb{Z}_N^{+},\forall k \in \mathbb{Z}_{T-1}^0:&\ \text{\eqref{eq:state_eq} holds},\\
				\forall i\in \mathbb{Z}_N^{+},\forall k \in \mathbb{Z}_T^{+} : &\ \text{\eqref{eq:output_eq} holds}, \\
				\forall l \in \mathbb{Z}_p^{1}, k\in \mathbb{Z}_T^{0}:&\ |z_{i,l}(k)-z_{j,l}(k)| \leq \delta^{\iota}, \label{separability condition}\\
				\forall   \bm{x}_0^\iota,y_T^\iota,d_T^{\iota},w_T^{\iota}, v_T^{\iota}:&\ \text{\eqref{eq:bmx0_polytope},\eqref{eq:y_polytope},\eqref{eq:d_polytope}-\eqref{eq:v_polytope} hold.} 
				\end{align}
			\end{subequations}
		\end{lemma}
		\begin{proof}
		Since the universal quantifier distributes over conjunction \cite[pp. 45--46]{rosen2011discrete}, the constraint \eqref{eq:separability_logic1} of Problem \ref{prob:1} is separated into two independent constraints for all possible values of the uncertain variables, i.e., the `responsibility' of the controlled input and the separation condition, respectively. 
		Moreover, to convert Problem \ref{prob:1} into the above bilevel optimization problem, we consider the double negation of the (non-convex) separability condition in \eqref{eq:separability_logic1} for each pair of models indexed by $\iota$. We first negate \eqref{eq:separability_logic1}, which is equivalent to the existence of uncertain variables such that the least upper bound on the difference $\delta^\iota$ between each component $l$ of the observed outputs for all time instance $k$, i.e., $|z_{i,l}(k)-z_{j,l}(k)| \leq \delta^{\iota}$, is equal to 0. Finally, to recover the original problem, once again we negate the above negation by enforcing that the minimum or least upper bound $\delta^\iota$ must be at least $\epsilon$, where $\epsilon$ is the amount of desired separation or simply the machine precision. In other words, we have a maximin game formulation: 
			\begin{align*}
			\forall \iota \in \mathbb{Z}_\mathrm{I}^{+} : \max_{k,l}\min_{\bm{x}_0^\iota,d_T^{\iota},w_T^{\iota}, v_T^{\iota}}\delta^\iota, \; 
			\text{s.t. \,} |z_{i,l}(k)-z_{j,l}(k)| \leq \delta^\iota, 
			\end{align*}
			where $\delta^\iota \geq \epsilon$ guarantees the desired model separation or discrimination (represented by \eqref{eq:eps} and \eqref{separability condition}).
		\end{proof}
		%
		
For the exact active model discrimination approach, we assume that the following holds:
		\begin{assumption} \label{assump:1}
		In the concatenated constraint of the `responsibility' of the uncontrolled input, i.e., $H_y^{\iota} \bar{x}^{\iota} \le \bar{p}_y^{\iota} - \bar{P}_y^{\iota} \tilde{f}_y^{\iota} - \bar{P}_y^{\iota} \Gamma_{yu}^{\iota} u_T$, $\bar{P}_y^{\iota} \Gamma_{yu} = 0$ is satisfied.
		\end{assumption}

Note that Assumption \ref{assump:1} ensures that the resulting optimization problem does not have bilinear terms. If Assumption \ref{assump:1} does not hold, the problem $(P_{DID})$ results in a mixed-integer nonlinear program (MINLP). A particular solution to this problem is provided in \cite{Harirchi2017}, where a sequence of restriction approach reduces this MINLP into a sequence of computationally tractable optimization problems.

		Next, leveraging Lemma \ref{prop:bilevel1}, we further recast the bilevel formulation as an MILP with SOS-1 constraints.
		
		\begin{theorem}[Discriminating Input Design as an MILP] \label{thm:nonconservative}
			Given a separability index $\epsilon$, the active model discrimination problem (Problem \ref{prob:1}) under Assumption \ref{assump:1} is equivalent 
			to the following mixed-integer optimization problem:
			
			{\small \vspace*{-0.4cm}\begin{subequations}
					\begin{align} \label{eq:exact}
					\tag{$P_{DID}$}&\qquad \min_{u_T,\delta^\iota,\bar{x}^\iota,\mu_1^\iota,\mu_2^\iota,\mu_3^\iota, \Pi^{\iota}} J(u_T) \hspace*{1cm}\\
					&\nonumber \begin{array}{rl}\text{s.t. \; }  \hspace*{1.0cm}
					&\bar{Q}_uu_T\leq\bar{q}_u,  \label{eq:u_concat}\\
					\forall \iota \in \mathbb{Z}_I^{+} :&  \Pi^{\iota\mathsf{T}} \begin{bmatrix} h_{\bar{x}}^{\iota} \\\bar{p}_y^\iota-\bar{P}_y^\iota\tilde{f}_y^\iota \end{bmatrix} \le  \bar{p}_x^\iota\hspace{-0.05cm}-\hspace{-0.05cm}\bar{P}_x^\iota\widetilde{f}_x^\iota\hspace{-0.05cm}-\hspace{-0.05cm}\bar{P}_{x}^{\iota} \Gamma_{xu}^{\iota} u_T	, \\ 
			\forall \iota \in \mathbb{Z}_I^{+} :&  \Pi^{\iota\mathsf{T}} \textstyle\diag{\{H_{\bar{x}}^{\iota},H_y^{\iota} \}}= H_{x}^{\iota}, \; \Pi^{\iota} \ge 0, \\
					\forall \iota \in \mathbb{Z}_I^{+} :& \ \delta^{\iota}(u_T) \geq \epsilon, \\
					\forall \iota \in \mathbb{Z}_I^{+} :& \ \eqref{kkt:start}-\eqref{kkt:end} \text{ hold,}\\
					\forall \iota \in \mathbb{Z}_I^{+}, \forall i \in \mathbb{Z}_{\kappa}^{+}:&\ \text{SOS-1}: \{ \mu_{1,i}^{\iota}, \widetilde{H}_{\bar{x},i}^{\iota}\bar{x}^{\iota}-h_{\bar{x},i}^{\iota}\},\\
					\forall \iota \in \mathbb{Z}_I^{+}, \forall j \in \mathbb{Z}_{\xi}^{+}:&\ \text{SOS-1}: \{ \mu_{2,j}^{\iota}, \widetilde{R}_j^{\iota} \bar{x}^{\iota}-r_j^{\iota}+\tilde{S}_j^{\iota} u_T\},\\
					\hspace*{-.2cm}\forall \iota \in \mathbb{Z}_I^{+}, \forall j \in \mathbb{Z}_{\xi+\rho}^{\xi+1}:&\ \text{SOS-1}:\{ \mu_{3,j}^{\iota}, \widetilde{R}_j^{\iota} \bar{x}^{\iota}-\delta^{\iota}-r_j^{\iota}+\tilde{S}_j^{\iota} u_T\},\end{array}
					\end{align}
				\end{subequations}}\normalsize
				where $\Pi^{\iota}$, $\mu_{1,i}^{\iota}$, $\mu_{2,j}^{\iota}$ and $\mu_{3,j}^{\iota}$ are dual variables, while $\widetilde{H}_{\bar{x},i}^{\iota}$, $\widetilde{R}_j^{\iota}$, $\widetilde{S}_j^{\iota}$ and equations \eqref{kkt:start}-\eqref{kkt:end}  are defined in the proof.
			\end{theorem}
			
			\begin{proof}
			First, in light of Assumption \ref{assump:1}, the `responsibility' of the controlled input in \eqref{eq:x_resp} can be formulated as its robust conterpart by introducing a dual matrix variable $\Pi^{\iota}$.
			
				Then, we rewrite the separability condition \eqref{separability condition} in the concatenated form as:
				\begin{align}
				\Lambda^\iota\bar{x}^\iota\leq\delta^\iota-\bar{E}^\iota\tilde{f}^\iota-(\bar{E}^\iota\Gamma_{u}^\iota+F_u^\iota)u_T,
				\end{align}
				where $\Lambda^\iota$,$\bar{E}^\iota$, $\tilde{f}^\iota$, $\Gamma_u^\iota$ and $F_u^\iota$ are matrices related to the separability condition that will be defined in the appendix. Moreover, we can concatenate the inequalities associated with $\bar{x}^\iota$, according to whether they are explicitly dependent on $u_T$ or not: 
				
				{\small \vspace*{-0.4cm}\begin{align}
					&R^{\iota}\bar{x}^{\iota} \leq \begin{bmatrix} \bm{0} \\ \mathbbm{1} \end{bmatrix}\delta^{\iota} + r^{\iota} - S^{\iota}u_T, \hspace{-0.3cm}& \ \text{[Explicitly dependent on $u_T$]}\hspace{-0.1cm}\\
					&H_{\bar{x}}^\iota \bar{x}^\iota \leq h_{\bar{x}}^\iota, \hspace{-0.3cm}&\text{[Implicitly dependent on $u_T$]}\hspace{-0.1cm} \label{eq:responsibilityd}
					\end{align}}\normalsize
				where we define
\begin{align*}
					\hspace{0cm} R^\iota \hspace{-0.075cm}=\hspace{-0.075cm} \begin{bmatrix} H_y^\iota\\\Lambda^\iota  \end{bmatrix}\hspace{-0.075cm},  r^\iota \hspace{-0.075cm}=\hspace{-0.075cm} \begin{bmatrix} \overline{p}_y^\iota-\overline{P}_y^\iota\tilde{f}_y^\iota\\-\bar{E}^\iota\tilde{f}^\iota \end{bmatrix}\hspace{-0.075cm},  S^\iota \hspace{-0.075cm}=\hspace{-0.075cm} \begin{bmatrix} \overline{P}_y^\iota\Gamma_{yu}^\iota\\\bar{E}^\iota\Gamma_{u}^\iota+F_u^\iota \end{bmatrix}. 
					\end{align*}
				
				Thus, the inner problem \eqref{eq:inner} for each $\iota \in I$ in Lemma \ref{prop:bilevel1} can be written in the concatenated form:
\begin{subequations}
						\begin{align} \label{eq:inner}
						\tag{$P_{Inner}^c$}\delta^{\iota*}(u_T)=&\min_{\delta^{\iota},\bar{x}^{\iota}}  \delta^{\iota} \\
						\text{s.t. \; }\hspace{1.3cm}
						& R^{\iota}\bar{x}^{\iota} \leq \begin{bmatrix} \bm{0} \\ \mathbbm{1} \end{bmatrix}\delta^{\iota} + r^{\iota} - S^{\iota}u_T, \label{inner:01}\\
						&H_{\bar{x}}^\iota \bar{x}^\iota \leq h_{\bar{x}}^\iota.
						\end{align}
					\end{subequations}

					A single level optimization can then be obtained by replacing the inner programs with their KKT conditions.
					
					{\small \vspace*{-0.4cm}\begin{subequations} \label{kkt}
							\begin{align}
							&0=\textstyle\sum_{i=1}^{i=\kappa} \mu_{1,i} ^{\iota}H_{\bar{x}}^{\iota}(i,m)+\textstyle\sum_{j=1}^{j=\xi} \mu_{2,j}^{\iota}R^{\iota}(j,m)\label{kkt:start}\\
							&\nonumber \qquad +\textstyle\sum_{j=\xi+1}^{j=\xi+\rho}\mu_{3,j}^{\iota}R^{\iota}(j,m), \quad  \forall m=1, \cdots, \eta,\\
							&0=1-{\mu_3^{\iota}}^\text{T}\mathbbm{1},\\
							&\widetilde{H}_{\bar{x},i}^{\iota}\bar{x}^{\iota}-h_{\bar{x},i}^{\iota}\leq0, \quad \forall i=1,\dots \kappa,\\
							&\widetilde{R}_j^{\iota} \bar{x}^{\iota}-r_j^{\iota}+S_j^{\iota} u_T\leq0, \quad\forall j=1,\dots \xi,\\
							&\widetilde{R}_j^{\iota} \bar{x}^{\iota}-\delta^{\iota}-r_j^{\iota}+S_j^{\iota} u_T\leq0, \quad\forall j=\xi+1,\dots \xi+\rho,\\
							&\mu_{1,i}^{\iota}\geq0, \quad\forall i=1,\dots \kappa,  \\
							&\mu_{2,j}^{\iota}\geq0, \quad\forall j=1,\dots \xi,  \\
							&\mu_{3,j}^{\iota}\geq0, \quad\forall j=\xi+1,\dots \xi+\rho,\label{kkt:end}\\
							&\mu_{1,i}^{\iota}(\widetilde{H}_{\bar{x},i}^{\iota}\bar{x}^{\iota}-h_{\bar{x},i}^{\iota})=0, \quad\forall i=1,\dots \kappa, \label{kkt:cs1} \\
							&\mu_{2,j}^{\iota}(\widetilde{R}_j^{\iota} \bar{x}^{\iota}-r_j^{\iota}+S_j^{\iota} u_T)=0, \quad\forall j=1,\dots \xi,  \\
							&\mu_{3,j}^{\iota}(\widetilde{R}_j^{\iota} \bar{x}^{\iota}-\delta^{\iota}-r_j^{\iota}+S_j^{\iota} u_T)=0,  \, \forall j=\xi+1,\dots \xi+\rho,\label{kkt:cs3}
							\end{align}
						\end{subequations}}\normalsize
						where $\widetilde{H}_{\bar{x},i}^{\iota}$ is the $i$-th row of ${H}_{\bar{x}}^{\iota}$, $\widetilde{R}_j^{\iota}$ and $\widetilde{S}_j^{\iota}$ are the $j$-th row of ${R}^{\iota}$ and $S^\iota$, respectively, $\eta=2(n+T(m_d+m_w+m_v))$ is the number of columns of $H_{\barr{x}}^\iota$, $\kappa=2(c_0+T(c_d+c_w+c_v))$ is the number of rows of $H_{\barr{x}}^\iota$, $\xi=2Tc_y$ is the number of rows of $\bm{0}$ in (\ref{inner:01})  and $\rho=2 Tp$ is the number of rows of $\mathbbm{1}$ in (\ref{inner:01}). 
						
						The bilinear constraints corresponding to the complementary slackness conditions \eqref{kkt:cs1}-\eqref{kkt:cs3} can then be enforced using SOS-1 constraints, which can readily be solved by Gurobi and CPLEX \cite{gurobi,cplex}. Finally, replacing \eqref{eq:u_polytope} with its concatenated form \eqref{eq:u_concat}, the MILP formulation follows. 
					\end{proof}
					
					\subsection{Conservative Active Model Discrimination Approach} \label{sec:conservative}
					
					Next, we summarize the proposed optimization formulation in \cite{harirchi2017active} to solve the conservative Problem \ref{prob:2}, against which we will compare our exact approach in the previous section in terms of computational complexity and optimality. This approach relies on formulating the problem as a robust optimization problem and further recasting the problem as an MILP as follows, whose proofs can be found in \cite{harirchi2017active}. 
					\begin{lemma}[Robust Optimization Formulation \cite{harirchi2017active}] \label{prop:robust1}
						Given a separability index $\epsilon$, the active model discrimination problem in Problem \ref{prob:2} is equivalent 
						to the following:
\begin{subequations} \label{eq:robust2}
								\begin{align}
								\tag{$P_{Robust}$}  	\nonumber \min_{u_T,s,a} & \ J(u_T) \\
								\text{s.t. }
								\label{eq:u}
								& \bar{Q}_u u_T \leq \bar{q}_u, \, R\bar{x} \leq r(u_T,s),\\
								& a\in\{0,1\}^{pTN(N-1)}, \, \textstyle\sum_{k,l,\alpha} a_{i,j,k,l,\alpha} \geq 1, \\
								& \text{SOS-1: } \{s_{i,j,k,l,\alpha}, a_{i,j,k,l,\alpha}\}, \\
								\label{eq:forall_cond}
								& \forall\bar{x}: \Phi\bar{x}\leq \psi(u_T),
								\end{align}
							\end{subequations}
							where $a$ is the vector of binary variables $a_{i,j,k,l,\alpha}$ concatenated over the indices in the order $i,j,k,l,\alpha$, and $s$ is similarly a vector of slack variables $s_{i,j,k,l,\alpha}$, defined as $s=\begin{bmatrix} s_1 \\ s_2 \end{bmatrix}$ where $s_\alpha$ for $\alpha \in \{1,2\}$ as well as $\bar{Q}_u$, $\Phi$, $R$, $\bar{q}_u$, $\psi(u_T)$ and $r(u_T,s)$ are defined 
							in the appendix. 
						\end{lemma}
						\begin{theorem}[Conservative Discriminating Input Design as an MILP \cite{harirchi2017active}]\label{thm:final_formulation}
							Given well-posed affine models and the separability index $\epsilon$, consider the following problem:
\begin{subequations} %
									\begin{align}
									&\qquad \min_{u_T,s,a,\Pi}  J(u_T) \label{eq:pcdid} \tag{$P_{CDID}$}\\
									&\nonumber \begin{array}{rl}\text{s.t. }
									& \bar{Q}_u u_T \leq \bar{q}_u, \Phi^T\Pi = R^T,  \Pi \geq 0,\\
									& \Pi^T\phi \leq r(u_T,s),  
									\\& a\in\{0,1\}^{pTN(N-1)}, \\
									&  \textstyle\sum_{k  \in \mathbb{Z}_T^{0}} \textstyle\sum_{l  \in \mathbb{Z}_p^{1}} \textstyle\sum_{\alpha\in\{1,2\}} a_{i,j,k,l,\alpha} \geq 1,\\
									& \text{SOS-1: } \{s_{i,j,k,l,\alpha}, a_{i,j,k,l,\alpha}\},\end{array}
									\end{align}
								\end{subequations}
								$\forall i,j  \in \mathbb{Z}_N^{+}: i < j$, $
								\forall l \in \mathbb{Z}_p^{1}$, $\forall k\in \mathbb{Z}_T^{0}$ and $\alpha\in\{1,2\}$, where $\Pi$ is a matrix of dual variables, while $\bar{Q}_u$, $\Phi$, $R$, $\bar{q}_u$, $\psi(u_T)$ 
								and $r(u_T,s)$ are problem-dependent matrices and vectors that are defined in the appendix, and $\phi$ is defined as the component-wise maximum of $\psi(u_T)$ where $u_T$ is subject to $u(k) \in \mathcal{U}$ for all $k \in \mathbb{Z}_{T-1}^0$. Then,
								\begin{enumerate}
									\item if {$\psi(u_T)$ is independent of $u_T$, i.e., $\psi(u_T)=\phi$,} 
									Problem \eqref{eq:pcdid} is equivalent up to the separability index $\epsilon$ to Problem \ref{prob:2} and its solution is optimal;
									\item if {$\psi(u_T)$ is dependent on $u_T$,}
									and if Problem \eqref{eq:pcdid} is feasible, {the solution of Problem \eqref{eq:pcdid} is sub-optimal with respect to Problem \ref{prob:2}.}
								\end{enumerate}
							\end{theorem}

\section{Complexity Analysis and Pair Elimination}\label{sec:complex}
To analyze the computational complexity of the exact and conservative active model discrimination formulations in Sections \ref{sec:exact} and \ref{sec:conservative}, we tabulate the number of SOS-1 constraints as well as binary and continuous  optimization variables for \eqref{eq:exact} and \eqref{eq:pcdid} (in Theorems \ref{thm:nonconservative} and \ref{thm:final_formulation}), respectively, as follows, with $c_1=n+T(m_d+m_w+m_v)$, $c_2=c_0+T(c_d+c_w+c_v)$ and $c_3=c_2+Tc_y$. 

\begin{table}[ht]
\caption{Complexity of Exact and Conservative Formulations \label{table:complex} } \vspace{-0.3cm}
\centering 
\begin{tabular}{ c  c  c  c } 
	\hline \multirow{ 2}{*}{Formulation} & \# of  &  \#  of  & \#  of  \\
	 & SOS-1 constr. & binary var. & continuous var. \\ 
	\hline \multirow{ 3}{*}{Exact}  & \multirow{ 3}{*}{$2I (c_3+Tp)$}  & \multirow{ 3}{*}{$0$} & $2I (c_3+Tp)+Tn_u$\\
	& & & $+2 I c_1 +I$\\
	& & & $+4 I T c_x c_3$\\ \hline
	 \multirow{ 3}{*}{Conservative} & \multirow{ 3}{*}{$2ITp$} & \multirow{ 3}{*}{$2ITp$} & $2ITp + T n_u$ \\
	& & & $+NT(c_2+T c_y)$ \\
	& & & $(Nc_x+2Ip)$\\
	\hline 
\end{tabular}
\label{table:complex1} 
\end{table}

Note that the SOS-1 constraints are integral constraints, which can be viewed as each contributing to an additional binary variable. Since in most active model discrimination problems, $c_3 \gg Tp$, from Table \ref{table:complex}, we observe that the exact formulation would involve many more {SOS-1 constraints,}
which are known to greatly increase the computational effort for solving an MILP. Thus, the exact formulation in Theorem \ref{thm:nonconservative} is typically more computationally complex than the conservative formulation in Theorem \ref{thm:final_formulation}, implying that  there is a trade-off between computational complexity and optimality. The exact formulation is optimal but complex, while the conservative formulation is suboptimal but computationally more tractable.

Furthermore, the computational complexity of both formulations can be seen to scale linearly with the number of model pairs $I$, i.e., factorially with the number of models $N$. Thus, to make the exact formulation in Theorem \ref{thm:nonconservative} more tractable, we propose to reduce the complexity of the formulation by decreasing $I$ via pair elimination\footnote{It is also possible to convert SOS-1 constraints into corresponding big-M formulations to potentially reduce their complexity but we found that this leads to unsurmountable numerical issues.}, described next.

\subsection{Pair Elimination}
Inspired by \cite{Scott2014Input}, we develop a preprocessing step to attempt to eliminate the model pairs that are trivially separated or discriminated with any controlled input $u$, i.e., all model pairs $\iota$ for which $\delta^{\iota*}(u_T)\geq\epsilon$ holds for all feasible $u_T$. This pair elimination approach can significantly reduce the cardinality of the set of model pairs $I$ that we need to consider when solving \eqref{eq:exact} in Theorem \ref{thm:nonconservative}, thus making the exact formulation much more tractable. 

More formally, if for the pair ($i$,$j$), indexed by $\iota$, the following condition holds, then we can eliminate the constraints corresponding to the pair $\iota$ in Theorem 1:
\begin{align}\label{eq:pair}
	\delta^{\iota*}(u_T)\geq\epsilon, \forall u_T: \overline{Q}_uu_T\leq\overline{q}_u.
\end{align}

In addition, it is straightforward to see that this robust optimization problem is equivalent to the feasibility problem of its negation:
\begin{subequations}\label{eq:pair2}
		\begin{align} \label{eq:Pelim}
		\tag{$P_{Elim}$} & \qquad \text{Find} \ u_T \hspace*{1cm}\\
		&\nonumber\begin{array}{rl} \text{s.t.}  \hspace*{0.1cm}
		&\bar{Q}_uu_T\leq\bar{q}_u, \\
	     & \delta^{\iota*}(u_T) < \epsilon. \end{array}
		\end{align}
\end{subequations}

Thus, if Problem $P_{Elim}$ is infeasible, then we can eliminate the constraints corresponding to pair $\iota$ in Theorem \ref{thm:nonconservative}.
\subsection{Numerical Example}
To illustrate and compare the difference between the complexity and optimality of  the exact and conservative formulations in Theorems \ref{thm:nonconservative} and \ref{thm:final_formulation}, we consider the following second-order linear system models:

\begin{align*}
&A_1=\begin{bmatrix} 0.6&0.2\\-0.4 &-0.2 \end{bmatrix},\ 	
B_1=\begin{bmatrix} 1&0\\0 &1\end{bmatrix},\ B_{w,1}=\begin{bmatrix} 1\\1 \end{bmatrix},\\
&C_1=\begin{bmatrix} 1&0\\0 &1 \end{bmatrix},\ 	
D_1=\begin{bmatrix} 0&0\\0 &0\end{bmatrix},\ D_{v,1}=\begin{bmatrix} 1\\1 \end{bmatrix},\\
&f_1=\begin{bmatrix} 0\\0 \end{bmatrix},g_1=\begin{bmatrix} 0\\0 \end{bmatrix} 
\end{align*}\normalsize
with four other models that are similarly defined as above with the following modifications:
\begin{align*}
&A_2=\begin{bmatrix} 1&0.2\\-0.4 &-0.2 \end{bmatrix},\ A_3=\begin{bmatrix} 0.6&-0.5\\-0.4 &-0.2 \end{bmatrix},\\
&\  B_4=\begin{bmatrix} 0&0\\0 &1\end{bmatrix}, \ C_5=\begin{bmatrix} 1&0\\0 &0 \end{bmatrix}.	
\end{align*}

The controlled input is bounded by {{$ -2\leq u(k)\leq2$}} and the uncertain variables for all $i=1,\cdots,5$ are also bounded by $0\leq x_i(0)\leq1$, $1\leq y_i(0)\leq2$, $-0.1\leq d_i(k)\leq0.1$, $-0.01\leq w_i(k)\leq0.01$ and $-0.01\leq v_i(k)\leq0.01$. \\

\begin{table}[h]
	\centering
	\caption{Results of numerical example} \vspace{-0.3cm}
	\label{table:numerical}
	\begin{tabular}{c c c c c }
		\hline\multicolumn{2}{c}{} & 1-norm & $\infty$-norm & 2-norm  \\ \hline
		\multirow{2}{*}{\eqref{eq:exact}}                                                                  & Optimal Value   & 0.074  & 0.074    & 0.00548 \\ 
		& Time (s) & 75.28  &   181.87       & 771.19  \\ \hline
		\multirow{2}{*}{\begin{tabular}[c]{@{}c@{}}\eqref{eq:exact} +\\  \eqref{eq:Pelim}\end{tabular}} & Optimal Value   & 0.074  & 0.074    & 0.00548 \\ 
		& Time (s) & 36.51  & 76.70    & 88.97   \\ \hline
		\multirow{2}{*}{\eqref{eq:pcdid}}                                                                  & Optimal Value   &      1.359  &       0.975   &     1.047    \\ 
		& Time (s) &    1.55&  1.57        &    1.22     \\ \hline
	\end{tabular}
\end{table}

In this example, we choose the 1-norm, $\infty$-norm and 2-norm as objective functions with $T=2$ and $\epsilon=0.01$. The optimal objective values and computation times\footnote{All the examples are implemented on a 3.1 GHz machine with 8 GB of memory running MacOS. 
	For the implementation of proposed approach, we utilized Yalmip \cite{YALMIP} and Gurobi \cite{gurobi} in the MATLAB environment.} for the exact formulation (cf. Theorem \ref{thm:nonconservative}) with and without pair elimination (i.e., \eqref{eq:exact} or \eqref{eq:exact}+\eqref{eq:Pelim}) as well as the conservative formulation \eqref{eq:pcdid} (cf. Theorem \ref{thm:nonconservative}) are shown in Table \ref{table:numerical}. Note that the depicted computation time for \eqref{eq:exact}+\eqref{eq:Pelim} is the combined sum of the all optimization routines, where we found that { 
	Model $5$ can be} eliminated based on the infeasibility of \eqref{eq:Pelim}.
	
	From Table \ref{table:numerical}, we observe that with the inclusion of the pair elimination procedure to the exact formulation \eqref{eq:exact}, the same optimal objective value can be obtained with significantly shorter computation time.
In addition,  {{although \eqref{eq:exact} obtained smaller objective values than \eqref{eq:pcdid}, its computation time was much longer.}} Furthermore, when the uncertainty sets become larger or the system dimension becomes higher, we can imagine that \eqref{eq:exact} will become even harder to solve. However, this is not a huge problem since the active model discrimination problem is solved offline. Nonetheless, there are still merits to the conservative solution using \eqref{eq:pcdid}, as a suboptimal solution that guarantees separations can be quickly found. In scenarios where computational resources are limited, then a trade-off between optimality and computational tractability will be necessary. {Moreover, if we change the input constraint to $-0.5 \leq u(k) \leq 0.5$, the feasible set of  \eqref{eq:pcdid} will become empty but \eqref{eq:exact} can still find the same optimal solution.}
\section{Application Case Studies: Active Intention Identification in (Semi-)Autonomous Vehicles} \label{sec:example}
In this section, we consider several case studies on active intention identification for (semi-)autonomous vehicles, where the underlying goal is to infer the intention of other road participants in hope of improving driving safety and performance. Specifically, we propose an active intention identification approach that consists of two components: (i) active model discrimination 
and (ii) model invalidation. 

The active model discrimination formulations (both exact and conservative) presented in the previous sections are solved \emph{offline} to design  optimal intention-revealing control inputs for the controlled/ego vehicle that can enhance the identification of the intention of other drivers from noisy sensor observations (by ``forcing" the intention models to behave differently), while optimizing for safety, comfort or fuel consumption, 
subject to engine power and braking limitations, as long as traffic laws are obeyed. Then, the obtained optimal input sequences will be implemented alongside a model invalidation algorithm (see \cite{Harirchi2015Model} for details) in \emph{real-time} to identify the intention of the other drivers  based on observed noisy output trajectories.


Our approach is general enough to capture various driving scenarios such as  intersection crossings and lane changes, provided that suitable models of intentions are available (from first principles or via data-driven approaches). For the sake of brevity and without loss of generality, we will only consider the scenarios of intersection crossing and lane changing, where the driver of the other vehicles, autonomous or human-driven, can choose among three intentions: inattentive, malicious and cautious (also used in \cite{Yong2014}), each of which represent nondeterministic sets of possible driving behaviors. The next two sections present the intention models for the intersection crossing and lane changing scenarios, respectively, followed by a section that discusses the obtained solutions for active model discrimination using both the exact and conservative formulations. 

\subsection{Intention Models for an Intersection Crossing Scenario}
\label{subsec:intersection}
We consider two\footnote{This is for ease of exposition. In fact, \emph{arbitrary} number of vehicles can be handled with the same input as long as their initial conditions are within a predefined set, $\mathcal{X}_0$ (e.g., within a certain distance before the intersection).} vehicles at an intersection (origin of coordinate system). The discrete-time equations of motion for these two vehicles are given by: 
\begin{align*}
x(k+1)&=x(k)+v_x(k) \delta t,\\
v_x(k+1)&=v_x(k)+u(k)\delta t+w_x (k) \delta t,\\
y(k+1)&=y(k)+v_y(k) \delta t,\\
v_y(k+1)&=v_y(k)+d_i(k) \delta t +w_y (k)\delta t,
\end{align*}
 where, respectively, $x$ and $y$ are ego car and other car's positions in $m$, $v_x$ and $v_y$ are ego car and other car's velocities in $\frac{m}{s}$, and $u$ and $d_i$ are ego car and other car's acceleration inputs in $\frac{m}{s^2}$, while $w_x$ and $w_y$ are process noise signals in $\frac{m}{s^2}$ and $\delta t$ is the sampling time in $s$. In this example,  $\delta t=0.3s$ and the acceleration input is given by $u(k) \in \mathcal{U} \equiv[-7.85,3.97]\frac{m}{s^2}$, where the maximum acceleration $u^{\text{max}}$ is calculated based on an acceleration of 0$-$100$\frac{km}{h}$ in 7s while the minimum acceleration $u^{min}$ corresponds to a maximum braking force of $0.8g$. 

 {Similar to \cite{Mamessier2014AIAA} and \cite{mcruer1974mathematical}, we model the human input (which reflects the human intention)  as a PD controller. The rest of model consists of simple vehicle dynamics.} We consider three driver intentions, $i \in \{ \textrm{I , C, M} \}$, corresponding to \underline{I}nattentive, \underline{C}autious and \underline{M}alicious drivers. 
 Using $\bm{\vec{x}}(k)=\begin{bmatrix} x(k), v_{x}(k), y(k), v_{y}(k) \end{bmatrix}^\intercal$ as the state vector, $\bm{\vec{u}}(k)=\begin{bmatrix} u(k), d_i(k) \end{bmatrix}^\intercal$ as the input vector and ${z}(k)=v_{y}(k) + v(k)$ as the noisy observation\footnote{Note that we assume the extreme scenario where only the other car's velocity is observed. Other sensor information, e.g., positions, can also be included and would only result in more optimal objective values.}, the vehicle dynamics for these driver intentions are modeled as the following:\\
\textbf{Inattentive Driver} ($i=\textrm{I}$), who is distracted and fails to notice the ego vehicle, thus attempting to maintain the velocity and proceeding with an uncontrolled input disturbance $d_{{I}}(k)\in\mathcal{D}_{{I}} \equiv 10\% \cdot \mathcal{U}$ (uncorrelated with $x(k)$ and $v_{x}(k)$):

{\small \vspace*{-0.4cm}\begin{align*} 
A_{I}&= \begin{bmatrix} 1 & \delta t & 0 & 0 \\ 0 & 1 & 0 & 0 \\ 0 & 0 & 1 & \delta t \\ 0 & 0 & 0 & 1 \end{bmatrix}\hspace{-0.05cm},\, B_{I}=\begin{bmatrix} 0 & 0 \\ \delta t & 0 \\ 0 & 0 \\ 0 & \delta t \end{bmatrix}\hspace{-0.05cm},\ B_{w,I}= B_{I},\\
C_{I}&=\begin{bmatrix} 0 &  0 & 0 & 1 \end{bmatrix},\,D_{I}=0, \,D_{v,I}=1.
\end{align*}}\normalsize
The inattentive vehicle also maintains forward mobility by ensuring that the velocity satisfies: $v_{y,I}(k) \in[6,9]\frac{m}{s}$.\\ 
\textbf{Cautious Driver} ($i=\textrm{C}$), who intends to stop at intersection with an input equal to $-K_{p,{{C}}} y(k)-K_{d,{{C}}} v_{y}(k)+d_{{C}}(k)$, where $K_{p,C}=1.5$ and $K_{d,C}=4.75$ are PD controller parameters. They are tuned to capture the characteristics of the cautious driver. We also allow an input uncertainty $d_{{C}}(k) \in \mathcal{D}_{{C}} \equiv 5\%\cdot \mathcal{U}$ to account for nonlinear, nondeterministic driving behaviors and heterogeneity between drivers: 

{\small \vspace*{-0.4cm}\begin{align*}
&A_{C}= \begin{bmatrix} 1 & \delta t & 0 & 0 \\ 0 & 1 & 0 & 0 \\ 0 & 0 & 1 & \delta t \\  0 & 0 & -K_p \delta t & 1-K_d \delta t \end{bmatrix}\hspace{-0.05cm},\, B_{C}=B_I, \\&B_{w,C}= B_{w,I},\
C_{C}=C_{I}, \,D_{C}=D_{I} \,D_{v,C}=D_{v,I}.
\end{align*}}\normalsize
\textbf{Malicious Driver} ($i=\textrm{M}$), who drives aggressively and attempts to cause a collision using an input equal to $K_{p,M }(x(k)-y(k))+K_{d,M } (v_{x}(k)-v_{y}(k))+d_M(k)$, where  $K_{p,C}=1$, $K_{d,C}=3.5$ are similarly PD controller parameters that are tuned to represent the characteristics of the aggressive driver and $d_M(k) \in \mathcal{D}_M \equiv 5\% \cdot \mathcal{U}$ is an input uncertainty to capture unmodeled variations among drivers:

{\small \vspace*{-0.4cm}\begin{align*}
&A_{M} \hspace{-0.05cm}=\hspace{-0.05cm} \begin{bmatrix} 1 & \delta t & 0 & 0 \\ 0 & 1 & 0 & 0 \\ 0 & 0 & 1 & \delta t \\  K_p \delta t & K_d \delta t & -K_p \delta t & 1-K_d \delta t \end{bmatrix}\hspace{-0.05cm}, B_{M}=B_I,\\&B_{w,M}=B_{w,I}, \
C_{M}=C_{I}, \,D_{M}=D_{I}, \,D_{v,M}=D_{v,I}.
\end{align*} }\normalsize

In addition, we choose the following initial conditions: 
\begin{align} \label{eq:initial states}
\begin{array}{c}
x(0) \in [15,18]m,\,  \quad v_{x}(0) \in [6,9]\frac{m}{s},\,\\   
y(0) \in  [15,18]m,\, \quad v_{y}(0) \in [6,9]\frac{m}{s},\,
\end{array}
\end{align}
where the initial velocities of the cars are constrained to match typical speed limits and the initial positions are based on reasonable distances from the intersection that still allow for a complete stop before the intersection, if needed. 
Moreover, the velocity of the ego car is constrained to be between $[0,9]\frac{m}{s}$ at all times 
to prevent it from moving backwards and from exceeding the speed limit. We also assume that the process and measurement noise signals are bounded with maximum magnitudes of 0.01 $\frac{m}{s^2}$ and 0.01 $\frac{m}{s}$, respectively, and the separability threshold $\epsilon$ is set to $0.25\frac{m}{s}$ to enforce a clear separation among intention models.

\subsection{Intention Models for a Lane Changing Scenario}
\label{subsec:lane-change}
Next, we describe the modeling assumptions as well as the intention models for a lane changing scenario on the highway. To simplify the problem, we assume that the other vehicle always drives in the center of its lane and hence has no motion in the lateral direction. We also assume that the lane width is $3.2m$. Under these assumptions, the discrete-time equations of motion for the ego and other vehicles are: 
		\begin{align*}
		\begin{array}{rl}
		x_e(k+1)&=x_e(k)+v_{x,e}(k) \delta t,\\
		v_{x,e}(k+1)&=v_{x,e}(k)+u(k) \delta t + w_{x,e} (k)\delta t,\\
		y_e(k+1)&=y_e(k)+v_{y,e}(k) \delta t + w_{y,e} (k)\delta t,\\
		x_{o}(k+1)&=x_o(k)+v_{x,o}(k) \delta t,\\
		v_{x,o}(k+1)&=v_{x,o}(k)+d_i(k) \delta t + w_{x,o} (k)\delta t,
		\end{array}
		\end{align*}
	where, respectively, $x_e$ and $y_e$, and $v_{x,e}$ and $v_{y,e}$ are the ego car's longitudinal and lateral positions in $m$, and the ego car's longitudinal and lateral velocities in $\frac{m}{s}$, $x_o$ and $v_{x,o}$ are the other car's longitudinal position in $m$ and longitudinal velocity in $\frac{m}{s}$, $u$ and $d_i$ are ego car and other car's acceleration inputs in $\frac{m}{s^2}$, $w_{x,e}$, $w_{x,e}$, $w_{x,e}$ are process noise signals in $\frac{m}{s^2}$ and $\delta t$ is the sampling time in $s$. For this example, we assume the following controlled inputs $u(k) \in \mathcal{U} \equiv[-7.85,3.97]\frac{m}{s^2}$ and $v_{y,e}(k) \in [-0.35,0]\frac{m}{s}$ (where $y$ is in the direction away from the other lane) and that the sampling time is  $\delta t=0.3s$.
As before, we consider three driver intentions $i \in \{ \textrm{I , C, M} \}$ that are modeled as: \\
	\textbf{Inattentive Driver} ($i=\textrm{I}$), who fails to notice the ego vehicle and tries to maintain his driving speed, thus proceeding with an acceleration input which lies in a small range $d_{{I}}(k)\in\mathcal{D}_{{I}} \equiv 10\%\cdot\mathcal{U}$ :

{\small \vspace*{-0.4cm}
\begin{align*} 
		&A_{I}= \begin{bmatrix} 1 & \delta t & 0 & 0 & 0\\ 0 & 1 & 0 & 0&0 \\ 0 & 0 & 1 &  0 & 0 \\ 0 & 0 & 0 & 1 & \delta t \\ 0 & 0 & 0 & 0 & 1 \end{bmatrix}\hspace{-0.05cm},\, B_{I}=\begin{bmatrix} 0 & 0 &0\\ \delta t & 0 &0\\ 0 & \delta t &0 \\0&0&0\\ 0 &0& \delta t \end{bmatrix}\hspace{-0.05cm},\,f_I=\bm{0}_{5\times1},\\
		&B_{w,I}=B_{I},
		C_{I}=\begin{bmatrix} 0&0&0&0&1	\end{bmatrix},\,D_{I}=0, \,D_{v,I}=1.
		\end{align*}}\normalsize
	\textbf{Cautious Driver} ($i=\textrm{C}$), who tends to yield the lane to the ego car with the input equal to $-K_{d,{{C}}} (v_{x,e}(k)-v_{x,o}(k))-L_{p,C}(\bar{y}-y_e(k))+L_{d,C}v_{y,e}(k)+d_{{C}}(k)$, where $K_{d,C}=0.9$, $L_{p,C}=2.5$ and $L_{d,C}=8.9$ are PD controller parameters, $\bar{y}=2$ and the input uncertainty is $d_{{C}}(k) \in \mathcal{D}_{{C}} \equiv 5\%\cdot\mathcal{U}$:

{\small \vspace*{-0.4cm}\begin{align*}
		&A_{C}= \begin{bmatrix} 1 & \delta t & 0 & 0 & 0\\ 0 & 1 & 0 & 0&0 \\ 0 & 0 & 1 &  0 & 0 \\ 0 & 0 & 0 & 1 & \delta t \\ 0 & -K_{d,{{C}}}\delta t & L_{p,C} \delta t& 0  & 1+K_{d,{{C}}}\delta t  \end{bmatrix},\\&B_{C}=\begin{bmatrix} 0 & 0 &0\\ \delta t & 0 &0\\ 0 & \delta t &0 \\0&0&0\\ 0 &L_{d,C}\delta t&\delta t \end{bmatrix}, f_{C}=\begin{bmatrix} 0\\ 0\\ 0 \\0\\ -L_{p,C}\bar{y}\delta t \end{bmatrix},\\&B_{w,C}=B_{w,I},\ C_{C}=C_{I}, D_{C}=D_{I}, D_{v,C}=D_{v,I}.
		\end{align*}}\normalsize \vspace{-0.5cm}
		
\noindent \textbf{Malicious Driver} ($i=\textrm{M}$), who does not want to yield the lane and  attempts to cause a collision with input equal to $K_{d,{{M}}} (v_{x,e}(k)-v_{x,o}(k))+L_{p,M}(\bar{y}-y_e(k))-L_{d,M}v_{y,e}(k)+d_{{M}}(k)$, if provoked, where $K_{d,M}=1.1$, $L_{p,M}=2.0$ and $L_{d,M}=8.7$ are PD controller parameters, $\bar{y}=2$ and the input uncertainty satisfies $d_M(k) \in \mathcal{D}_M \equiv 5\%\cdot\mathcal{U}$:

{\small \vspace*{-0.4cm} \begin{align*}
		&A_{M}= \begin{bmatrix} 1 & \delta t & 0 & 0 & 0\\ 0 & 1 & 0 & 0&0 \\ 0 & 0 & 1 &  0 & 0 \\ 0 & 0 & 0 & 1 & \delta t \\ 0& K_{d,{{M}}}\delta t & -L_{p,M} \delta t&0 & 1-K_{d,{{M}}} \delta t \end{bmatrix},\\&B_{M}=\begin{bmatrix} 0 & 0 &0\\ \delta t & 0 &0\\ 0 & \delta t &0 \\0&0&0\\ 0 &-L_{d,M} \delta t & \delta t \end{bmatrix} ,f_{M}=\begin{bmatrix} 0\\ 0\\ 0 \\0\\ L_{p,C}\bar{y}\delta t \end{bmatrix},\\&B_{w,M}= B_{w,I}, C_{M}=C_{I},D_{M}=D_{I}, D_{v,M}=D_{v,I}. 
		\end{align*} }\normalsize \vspace{-0.5cm}
		
	\indent  Without loss of generality, we assume that the initial position of the ego car is 0, and the initial position of the other car is constrained by their initial relative distance. The initial velocities of the cars are also constrained to match typical speed limits of the highway. Further, we assume that at the beginning, both cars are close to the center of the lanes. In this case, the initial conditions are as follows: 
\begin{align} \label{eq:initial states-lane}
		\begin{array}{c}
		v_{x,e}(0) \in [30,32]\frac{m}{s} \,\quad 	y_e(0) \in  [1.1,1.8]m\\   
		v_{x,o}(0) \in [30,32]\frac{m}{s},\quad x_o(0) \in  [7,12]m \, \quad \,
		\end{array}
		\end{align}
	
	Moreover, the velocity of the ego vehicle is constrained between $[27,35]\frac{m}{s}$ at all times to obey the speed limit of a highway and the lateral position of the ego vehicle is constrained between $[0.5,2]m$. Process and measurement noise signals are also limited to the range of $[-0.01 , 0.01]$ and the separability threshold is set to $\epsilon=0.25\frac{m}{s}$.

\subsection{Simulation Results and Discussions} 

In this section, we demonstrate the effectiveness of the proposed active intention approach for both of the above-described intersection crossing and lane changing scenarios. In the intersection crossing scenario (cf. Fig. \ref{fig:intersection}), the driver of car $1$ is cautious and intends to stop at the intersection, car $2$ is malicious and attempts to match the ego car's velocity to cause a collision,  while car $3$ corresponds to an inattentive driver who fails to notice the ego vehicle and instead maintains his/her velocity. In this case, the ego car probes the intentions of the other cars by speeding up or slowing down and observing their responses.

On the other hand, in the lane changing scenario (cf. Fig. \ref{fig:lanechange}), the ego car actively nudges into the other lanes in order to discern the other cars' intention based on their responses. In this simulation example, car $1$ is inattentive, car $2$ is cautious and car $3$ is malicious and refuses to yield. 

In both scenarios, the \emph{offline} implementation of our active model discrimination approach yields an optimal separating input that guarantees that the velocities of the other vehicles are different by at least $\epsilon$ under each intention. Then, the optimal separating input is applied in \emph{real-time} and the intention is identified using the model invalidation approach in \cite{Harirchi2015Model}. Fig.  \ref{fig:animation} shows that the active intention identification procedure indeed succeeds at inferring other cars' intentions. 

\begin{figure}[t!]
	\begin{center}
		\subfigure[Intersection crossing scenario. \label{fig:intersection}]{\includegraphics[width=3.27in,trim=2mm 0mm 1mm 0mm]{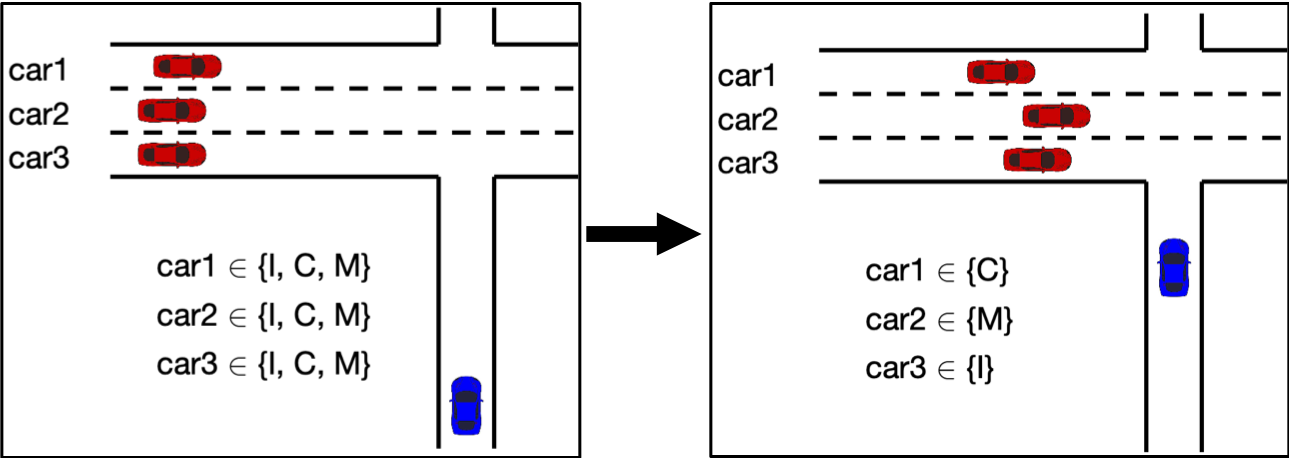}} \\
	\subfigure[Lane changing scenario. \label{fig:lanechange}]
	{\includegraphics[width=3.35in,trim=0mm 5mm 0mm 3mm]{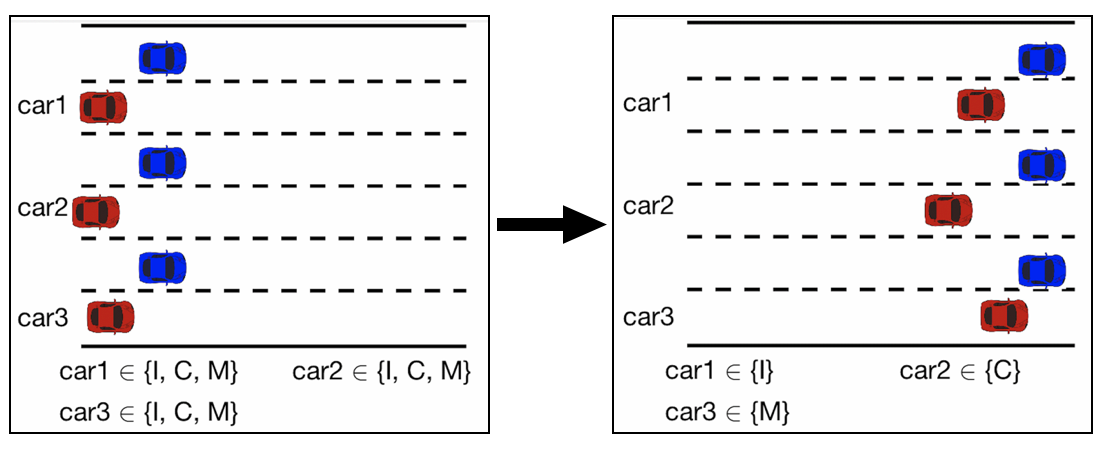}}
		\caption{Snapshots of an animation of the active intention identification approach when implementing the separating input obtained from active model discrimination and the model invalidation algorithm \cite{Harirchi2015Model} for intention inference (cf. \protect\url{https://youtu.be/U79-pjXmTWc} for the full animation video). \label{fig:animation} }
	\end{center}\vspace{-0.2cm}
\end{figure}

\begin{figure}[h!]
	\vspace{-0cm}
	\begin{center}
		\includegraphics[width=3.5in,trim=0mm 15mm 1mm 15mm]{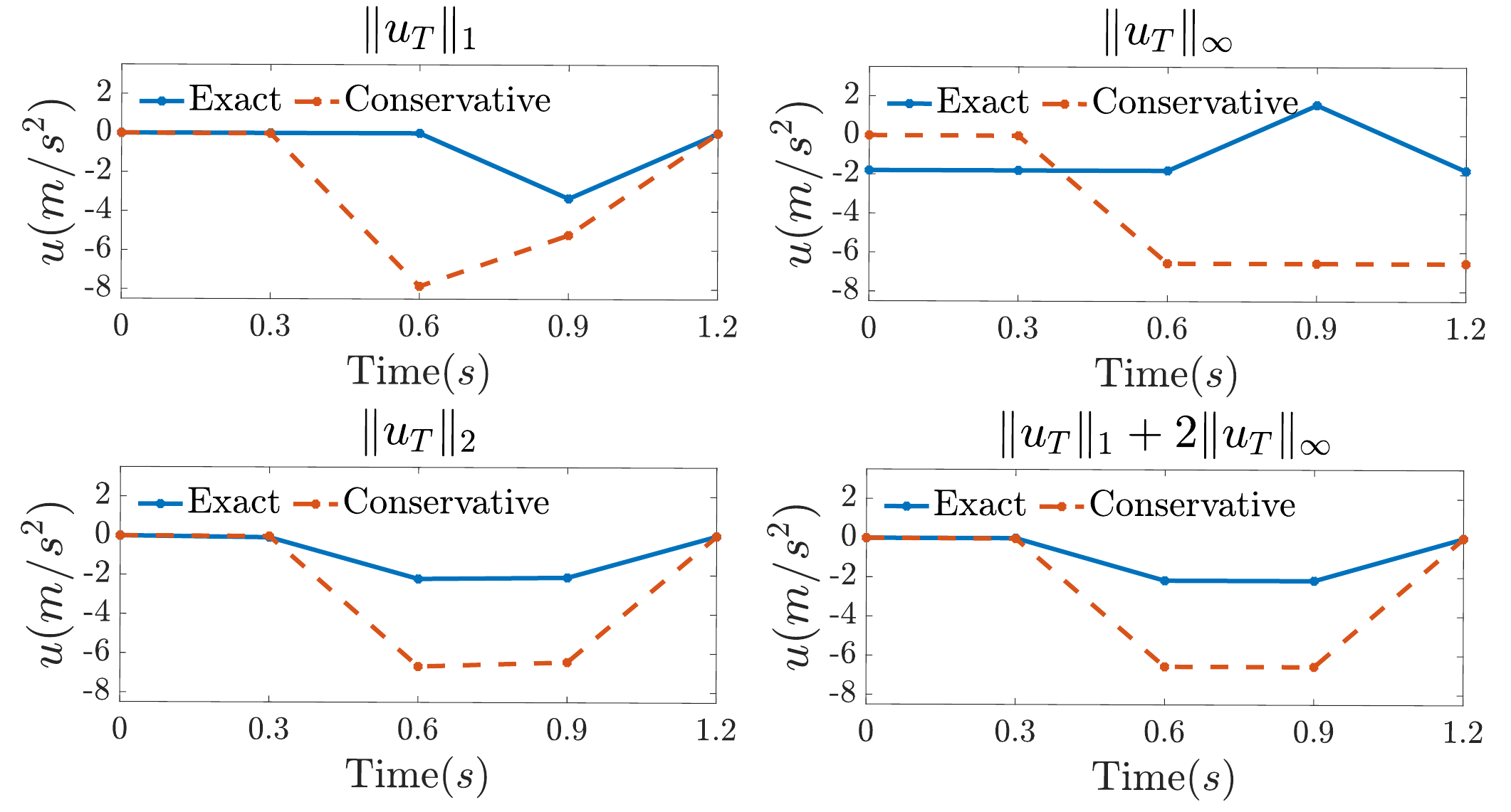} 
		\caption{Intersection crossing scenario: Comparison of optimal separating inputs $u(\frac{m}{s^2})$ with various cost functions.  \label{fig:int_obj} }
	\end{center}
	\begin{center}
	\subfigure[Trajectories of the controlled input $u(\frac{m}{s^2})$.]{
\centering
		\includegraphics[width=3.4in,trim=5mm 7mm 1mm 7mm]{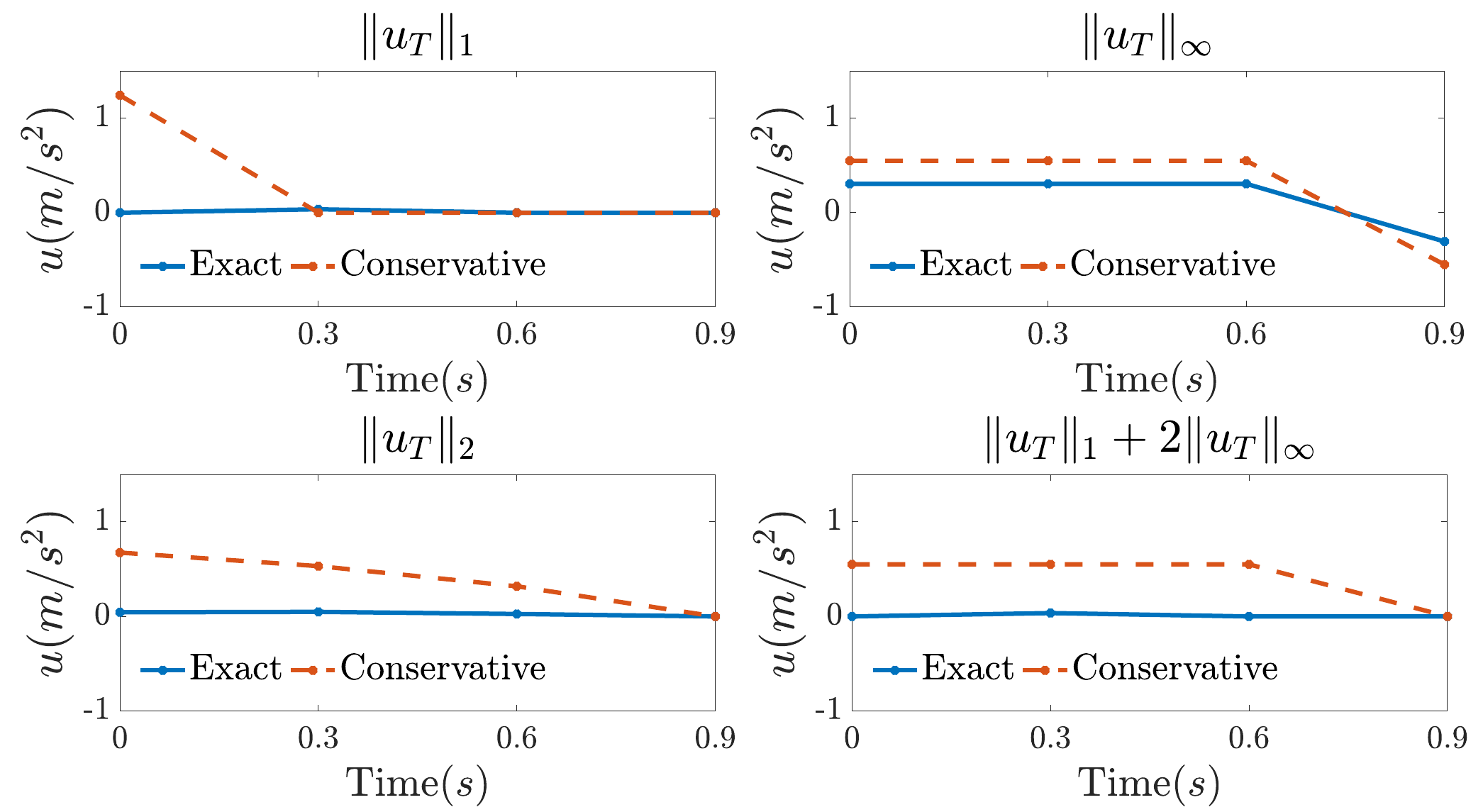}
}
	\subfigure[Trajectories of the controlled input $v_{y,e}(\frac{m}{s})$.]{
\centering
	    \includegraphics[width=3.4in,trim=5mm 7mm 1mm 6mm]{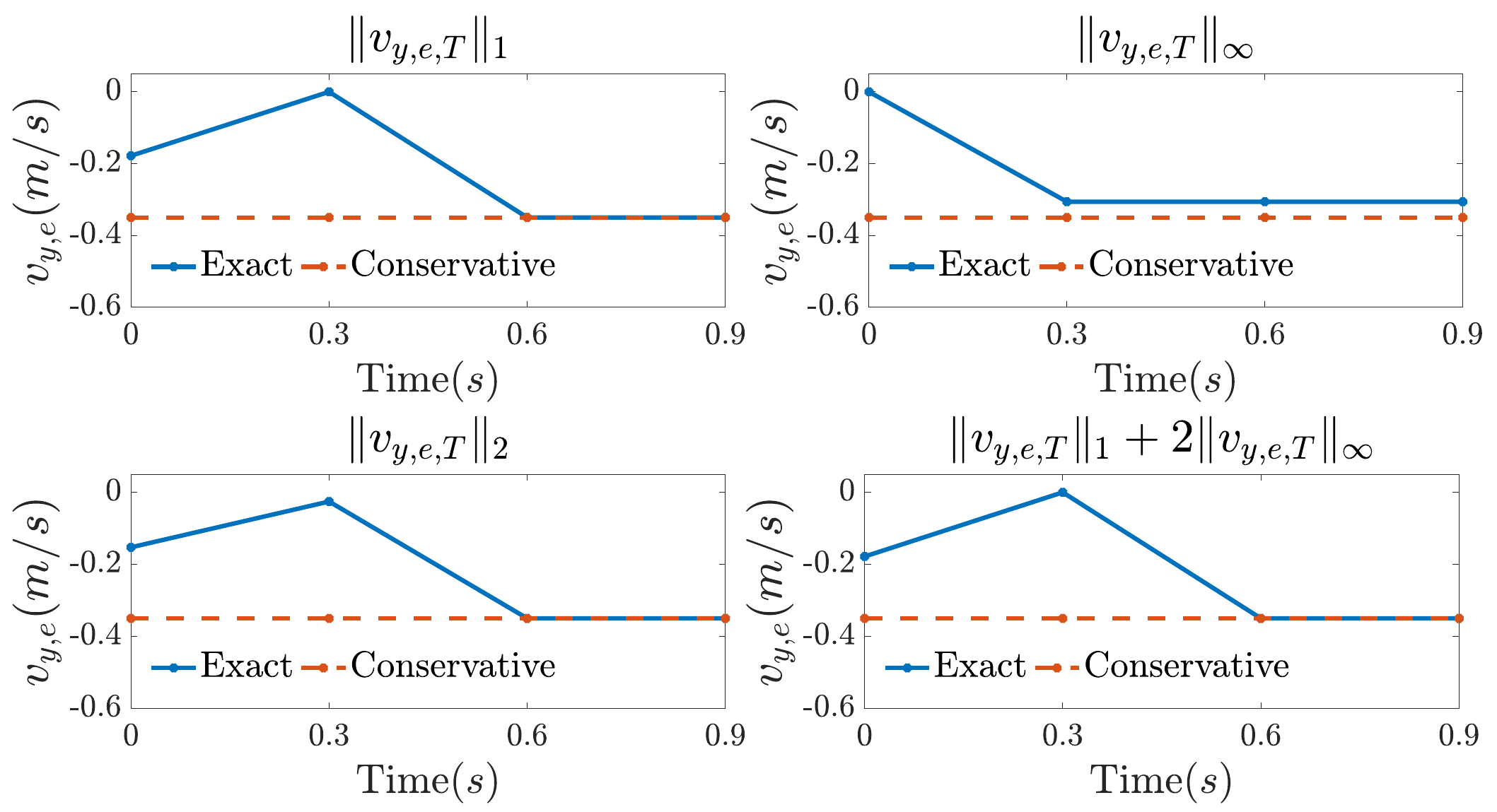} 
}\vspace{-0.15cm}
		\caption{Lane changing scenario: Comparison of optimal separating inputs with various cost functions. \label{fig:lane_obj} }
	\end{center}\vspace{-0.2cm}
\end{figure}

Moreover, we compare the exact and conservative formulations using both scenarios and demonstrate below that the conservative formulation is suboptimal but scalable, while the exact formulation is optimal but computationally costly: 
\subsubsection{\textbf{Optimal separating input for various cost functions}} 
In this case study, we consider several convex cost functions involving the input sequence, including $\|u_{T}\|_1$ that enforces sparsity (leads to minimal number of non-zero inputs), $\|u_{T}\|_2$ that minimizes fuel consumption (smooth acceleration/braking) and $\|u_{T}\|_\infty$ that ensures comfort (with small maximum input amplitudes). We also consider a combination of $\|u_{T}\|_1$ and $\|(\Delta u)_{T}\|_\infty$ that trades-off between sparsity and comfort, where $(\Delta u)_{T}$ is the rate of change in the inputs,  defined as $(\Delta u)_{T} = [u(1)-u(0) \;\; \hdots \;\;  u(T-1)-u(T-2)]^\intercal$. 

We first compare the exact and conservative formulations in the intersection crossing scenario. As seen in Fig. \ref{fig:int_obj}, the separating input designed by the exact formulation (solid line) has a smaller amplitude than the separating input of the conservative formulation (dashed line), implying that exact formulation is the more optimal solution. 
In addition, we observe that both the exact formulation and the conservative formulation design inputs to decelerate the ego car for intention identification.  

Similarly, in the lane changing scenario, we observe that the exact formulation gives a better solution, as expected (cf. Fig. \ref{fig:lane_obj}). In order to discriminate the other driver's intention, the separating input obtained by the conservative formulation requires the ego vehicle to nudge into the other driver's lane with maximum lateral velocity during the entire time horizon, which is very aggressive and dangerous. By contrast, 
the separating input obtained by the exact formulation  is relatively safer with  lateral nudging for shorter time intervals and with smaller lateral velocity. 

%

\subsubsection{\textbf{Trade-off between optimal value and computation time}} As shown in Tables \ref{table:lnt_valuetime} and \ref{table:lane_valuetime}, the exact formulation achieves better objective values at the cost of increased computation time in both driving scenarios we considered. As the primary goal of active intention identification, is to find an intention-revealing input to distinguish among potential driver intentions \emph{offline}, a larger computation time is not a critical issue. Nonetheless, the conservative formulation is still relevant in the cases when online active model discrimination can further improve the separating input when more information becomes available.  

\begin{table}[ht]
	\renewcommand{\arraystretch}{1} 
	\caption{Intersection Crossing Scenario}  \vspace*{-0.2cm} 
	\centering 
	\begin{tabular}{C{2cm} C{1cm} c C{1cm} c }
		\hline
		\multirow{2}{*}{}        & \multicolumn{2}{c }{Exact} & \multicolumn{2}{c}{Conservative} \\ \cline{2-5} 
		Cost Function & Optimal Value        & Time($s$)       & Optimal Value         & Time($s$)      \\ \hline
		$||u||_1$                 & 3.374     &24.89        & 13.108     & 2.75          \\ 
		$||u||_2$                 & 3.053      & $>$5000        & 9.271     & 2.84        \\ 
		$||u||_{\infty}$          & 1.804     &99.75 & 6.556      & 3.40       \\ 
		$||u||_1+2||u||_{\infty}$ & 8.660         & 85.43        &26.267   & 2.74       \\ \hline
	\end{tabular}
	\label{table:lnt_valuetime} 
\end{table}


\begin{table}[ht]
		\renewcommand{\arraystretch}{1} 
	\caption{Lane Changing Scenario}  \vspace*{-0.2cm}
	\centering 
	\begin{tabular}{ C{2cm} C{1cm} c C{1cm} c }
		\hline
		\multirow{2}{*}{}         & \multicolumn{2}{c }{Exact} & \multicolumn{2}{c }{Conservative} \\ \cline{2-5} 
		Cost Function& Optimal Value        & Time($s$)       & Optimal Value         & Time($s$)      \\ \hline
		$||u||_1$                 		& 0.914     	& 138.77       &2.645      	&   3.21        \\ 
		$||u||_2$                 		&  0.523     	& 3435.31     & 1.153    	&  3.09       \\ 
		$||u||_{\infty}$          		& 0.306 		&  140.22      & 0.551    	&  2.72    \\ 
		$||u||_1+2||u||_{\infty}$ 	& 1.614           	& 69.72         & 4.155         & 3.23       \\ \hline
	\end{tabular}
	\label{table:lane_valuetime} 
\end{table}

\section{Conclusion} \label{sec:conclude}

%

In this paper, a novel model-based optimization approach is proposed to find a safe and optimal 
discriminating input, which guarantees the distinction among multiple affine models with uncontrolled inputs and noise. This approach improves on a previous robust optimization formulation that turned out to be a feasible but conservative solution at the cost of increased computational complexity.
The new active model discrimination problem can be recast as a bilevel optimization problem and subsequently as a tractable MILP problem  with SOS-1 constraints.
To illustrate the efficiency of this framework, we successfully applied the proposed active model discrimination approach to the problem of intention identification of other human-driven or autonomous vehicles in scenarios of intersection crossings and lane changing. {In the future, we plan to develop methodologies to learn intention models from data and then apply the proposed active model discrimination approach to distinguish the learned intention models.} 
\section*{Acknowledgment}

The authors would like to thank Dr. Qiang Shen from Arizona State University for 
pointing out and correcting an omission in \eqref{eq:20}, and his help in updating the manuscript (including the addition of Assumption \ref{assump:1}) and simulations accordingly.

\balance

\bibliographystyle{unsrt}
\bibliography{biblio,fault}

\appendix \label{appendix}

In this appendix, we provide definitions of matrices and vectors that were previously omitted to improve readability.
\subsection{Time-Concatenated Matrices and Vectors in Section \ref{sec:model}:}
\vspace{-0.5cm}
\begin{align}
\nonumber&\barr{A}_{i,T} = \begin{bmatrix}
A_{i} \\
A_{i}^2 \\
\vdots \\
A_{i}^T
\end{bmatrix}\hspace{-0.05cm}, \quad
\Theta_{i,T} = \begin{bmatrix}
\mathbbm{I}& 0 & \dotsm & 0 \\
A_{i} & \mathbbm{I} & \dotsm & 0 \\
\vdots && \ddots & \\
A_{i}^{T-1} & A_{i}^{T-2} & \dotsm & \mathbbm{I}
\end{bmatrix}, \\
\nonumber& \barr{f}_{i,T}=\displaystyle\vect_T\{f_i\},\  \tilde{f}_{i,T}=\Theta_{i,T} \barr{f}_{i,T}, \ \tilde{g}_{i,T} = \displaystyle\vect_T \{ g_i \},  \\
\nonumber&E_i = \displaystyle\diag_T \{ C_i\}, \ F_{u,i} = \displaystyle\diag_T \{ D_{u,i}\},\\ 
\nonumber & F_{d,i} = \displaystyle\diag_T \{ D_{d,i}\}, \ F_{v,i} = \displaystyle\diag_T \{ D_{v,i}\}. \\
\nonumber & \text{For }\dagger=\{x,y\} \text{ and } \star=\{u,d,w\}:\\
\nonumber& B_{\star,i}=\begin{bmatrix} B_{x\star,i} \\ B_{y\star,i} \end{bmatrix}, \  B_{\dagger \star,\text{d},i,T}=\displaystyle\diag_T\{ B_{\dagger \star,i}\},\\
&\nonumber \Gamma_{\star,i,T} \hspace{-0.05cm}=\hspace{-0.05cm} \begin{bmatrix}
B_{\star,i} & 0 & \dotsm & 0 \\
A_{i}B_{\star,i} & B_{\star,i} & \dotsm & 0 \\
\vdots && \ddots & \\
A_{i}^{T-1}B_{\star,i} & A_{i}^{T-2}B_{\star,i} & \dotsm & B_{\star,i}
\end{bmatrix}\hspace{-0.05cm},\\
 &\nonumber A_{\dagger,\text{d},i,T}=\displaystyle\diag_T\{\begin{bmatrix} A_{\dagger x,i} & A_{\dagger y,i} \end{bmatrix}\},  \\
\nonumber& M_{\dagger,i,T} = A_{\dagger,\text{d},i,T}\begin{bmatrix} \mathbbm{I} \\ \barr{A}_{i,T-1} \end{bmatrix}, \ \barr{f}_{\dagger,i,T}= \vect_T \{f_{\dagger,i}\},\\
\nonumber&
\tilde{f}_{\dagger,i,T} = A_{\dagger,\text{d},i,T}\begin{bmatrix} 0 \\ \Theta_{i,T-1} \end{bmatrix}\barr{f}_{i,T-1} + \barr{f}_{\dagger,i,T}, \\ 
\nonumber&  \Gamma_{\dagger \star,i,T} = A_{\dagger,\text{d},i,T}\begin{bmatrix}
0 & 0 \\
\Gamma_{\star,i,T-1} & 0
\end{bmatrix} + B_{\dagger \star,\text{d},i,T}.
\end{align}


\subsection{Matrices and Vectors in Theorem \ref{thm:nonconservative}:}
\vspace{-0.5cm}
\begin{align*}
&\barr{A}^\iota=\displaystyle\diag_{i,j} \{\barr{A}_{i,T}\},  \   \barr{C}^\iota = \displaystyle\diag_{i,j}  \{ E_i\}, \ \tilde{f}^\iota =\displaystyle\vect_{i,j} \{\tilde{f}_{i,T}\},  \\
&  \tilde{g}^\iota = \displaystyle\vect_{i,j} \{\tilde{g}_{i,T}\}, \ \Gamma_u^\iota = \displaystyle\vect_{i,j} \{\Gamma_{u,i,T}\}, \\
&    \Gamma_d^\iota = \displaystyle\diag_{i,j}N \{\Gamma_{d,i,T}\}, \ \Gamma_w^\iota=\displaystyle\diag_{i,j} \{\Gamma_{w,i,T}\},  \\
&\nonumber    \barr{D}_u^\iota = \vect_{i,j} \{ F_{u,i}\}, \ \barr{D}_d ^\iota= \diag_{i,j} \{ F_{d,i}\}, \ \barr{D}_v^\iota = \diag_{i,j}  \{ F_{v,i}\},\\
& \Lambda ^\iota= \barr{E}^\iota\begin{bmatrix} \barr{A}^\iota & \Gamma_d ^\iota& \Gamma_w ^\iota& \bm{0} \end{bmatrix} + \begin{bmatrix} \bm{0}  & \barr{F}_d^\iota &\bm{0}  & \barr{F}_v^\iota \end{bmatrix},\\
&\barr{E}^\iota = \begin{bmatrix}
E_i & -E_j \\
-E_i & E_j  \\
\end{bmatrix}.\\
& \text{For }\dagger=\{x,y\}:\\
& \Gamma_{\dagger u}^\iota =\vect_{i,j} \{\Gamma_{\dagger u,i,T} \}, \ \Gamma_{\dagger d}^\iota =\diag_{i,j} \{\Gamma_{\dagger d,i,T} \}, \\
& 
\Gamma_{\dagger w}^\iota=\displaystyle\diag_{i,j} \{\Gamma_{\dagger w,i,T}\}, \ 
M_\dagger^\iota=\diag_{i,j} \{M_{\dagger,i,T}\},\\
& \tilde{f}_{\dagger} ^\iota=\vect_{i,j} \{\tilde{f}_{\dagger,i,T}\}. \\
\nonumber & \text{For }\ast=\{d,v\}:\\
&\barr{F}_\ast^\iota  = \begin{bmatrix}
F_{\ast,i} & -F_{\ast,j} \\
-F_{\ast,i} & F_{\ast,j}  
\end{bmatrix}, \ \barr{F}_u^\iota = \begin{bmatrix}
F_{u,i} -F_{u,j}\\
F_{u,j} -F_{u,i}
\end{bmatrix}, \\
&  \barr{g}^\iota = \begin{bmatrix}
\tilde{g}_i - \tilde{g}_j \\
-\tilde{g}_i + \tilde{g}_j 
\end{bmatrix}.  
\end{align*}

\subsection{Matrices and Vectors in Theorem \ref{thm:final_formulation}:}
\vspace{-0.4cm}
\begin{align*}
&\barr{A}=\displaystyle\diag_{i=1}^N \{\barr{A}_{i,T}\},  \   \barr{C} = \displaystyle\diag_{i=1}^N  \{ E_i\}, \ \tilde{f} =\displaystyle\vect_{i=1}^N \{\tilde{f}_{i,T}\},    \\
&  \tilde{g} = \displaystyle\vect_{i=1}^N \{\tilde{g}_{i,T}\}, \
 \Gamma_u = \displaystyle\vect_{i=1}^N \{\Gamma_{u,i,T}\}, \\
&   \Gamma_d = \displaystyle\diag_{i=1}^N \{\Gamma_{d,i,T}\},  \ \Gamma_w=\displaystyle\diag_{i=1}^N \{\Gamma_{w,i,T}\},  \\
&\nonumber    \barr{D}_u = \vect_{i=1}^N \{ F_{u,i}\}, \ \barr{D}_d = \diag_{i=1}^N \{ F_{d,i}\}, \ \barr{D}_v = \diag_{i=1}^N  \{ F_{v,i}\},\\
& \Lambda = \barr{E}\begin{bmatrix} \barr{A} & \Gamma_v & \Gamma_w & \bm{0} \end{bmatrix} + \begin{bmatrix} \bm{0}  & \barr{F}_v &\bm{0}  & \barr{F}_v \end{bmatrix}, \\
& \lambda(u_T,s) = \epsilon\mathbbm{1} - s - \barr{g} - (\barr{E}\Gamma_u + \barr{F}_u) u_T - \barr{E}\tilde{f}, \\
	&R \hspace{-0.0cm}=\hspace{-0.0cm} \begin{bmatrix} -\Lambda \\ H_x \end{bmatrix}\hspace{-0.05cm}, \ r(u_T,s) \hspace{-0.0cm}=\hspace{-0.0cm} \begin{bmatrix} -\lambda(u_T,s) \\ h_x(u_T) \end{bmatrix}\hspace{-0.05cm}, \
	\Phi \hspace{-0.0cm}=\hspace{-0.0cm} \begin{bmatrix} H_y \\ H_{\bar{x}} \end{bmatrix}\hspace{-0.05cm},\\
	&  \psi(u_T) \hspace{-0.075cm}=\hspace{-0.075cm} \begin{bmatrix} h_y(u_T) \\ h_{\bar{x}} \end{bmatrix},\\
	 &E = \begin{bmatrix}
E_1 & -E_2 & 0 & \dotsm & \dotsm & 0 \\
E_1 & 0 & -E_3 & 0 & \dotsm & 0 \\
\vdots \\
0 & \dotsm & \dotsm & 0 & E_{N-1} & -E_{N} \\
\end{bmatrix},\\
\nonumber& \barr{E}= \begin{bmatrix} E \\ -E \end{bmatrix}, \quad \barr{F}_u= \begin{bmatrix} F_u \\ -F_u \end{bmatrix}, \quad \barr{F}_d= \begin{bmatrix} F_d \\ -F_d \end{bmatrix},\\ &\barr{F}_v= \begin{bmatrix} F_v \\ -F_v \end{bmatrix}, \quad \barr{g}= \begin{bmatrix} g \\ -g \end{bmatrix},\\
& F_u = \begin{bmatrix}
F_{u,1} -F_{u,2}\\
F_{u,1}  -F_{u,3} \\
\vdots \\
F_{u,N-1}  -F_{u,N} \\
\end{bmatrix}\hspace*{-0.0cm}, \ g = \begin{bmatrix}
\tilde{g}_1 - \tilde{g}_2 \\
\tilde{g}_1 - \tilde{g}_3 \\
\vdots \\
\tilde{g}_{N-1} - \tilde{g}_N
\end{bmatrix},  \\
\nonumber &s_{i,j,\alpha}=\begin{bmatrix}\vect_{l=1}^p \{s_{i,j,1,l,\alpha}\}  \\ \vdots \\ \vect_{l=1}^p \{s_{i,j,T,l,\alpha}\}\end{bmatrix},
\ 
s_\alpha= \begin{bmatrix} \vect_{j=2}^N \{s_{1,j,\alpha}\} \\ \vect_{j=3}^N \{s_{2,j,\alpha}\}  \\ \vdots  \\ s_{N-1,N,\alpha} \end{bmatrix}.\\
& \text{For }\dagger=\{x,y\}:\\	
&\Gamma_{\dagger u} =\vect_{i=1}^N \{\Gamma_{\dagger u,i,T} \}, \  \Gamma_{\dagger d} =\diag_{i=1}^N \{\Gamma_{\dagger d,i,T} \}, 
\\
& \Gamma_{\dagger w}=\displaystyle\diag_{i=1}^N \{\Gamma_{\dagger w,i,T}\}, \ M_\dagger=\diag_{i=1}^N \{M_{\dagger,i,T}\},\\
& \tilde{f}_{\dagger} =\vect_{i=1}^N \{\tilde{f}_{\dagger,i,T}\}.\\
\nonumber & \text{for }\ast=\{d,v\}:\\
&F_\ast = \begin{bmatrix}
F_{\ast,1} & -F_{\ast,2} & 0 & \dotsm & \dotsm & 0 \\
F_{\ast,1} & 0 & -F_{\ast,3} & 0 & \dotsm & 0 \\
\vdots \\
0 & \dotsm & \dotsm & 0 & F_{\ast,N-1} & -F_{\ast,N} \\
\end{bmatrix}.
\end{align*}

\end{document}